%% file: multisphere.tex
\documentclass{article}
\usepackage{amsmath, amssymb, amsthm, bm, stmaryrd}
\usepackage[utf8]{inputenc}
\usepackage[T1]{fontenc}
\usepackage{lmodern}
\usepackage{hyperref,cleveref}
\usepackage{todonotes}

\DeclareMathOperator{\sos}{sos}
\DeclareMathOperator{\mom}{mom}
\DeclareMathOperator{\rank}{rank}

\DeclareMathOperator{\vecc}{vec}
\DeclareMathOperator{\QM}{QM}
\DeclareMathOperator{\st}{s.t.}
\DeclareMathOperator{\Span}{span}

\newcommand{\di}{\mathrm{d}}
\newcommand\R {\mathbb{R}}
\newcommand\N {\mathbb{N}}

\newcommand\sph {\mathbb{S}}
\newcommand{\sphs}{\mathcal{S}}
\newcommand{\manif}{\mathcal{K}}

\newcommand{\pn}{\bm{\bar{p}}}
\newcommand{\rkone}{\mathcal{Q}}
\newcommand{\rkonex}{\mathcal{X}}

\newtheorem{thm}{Theorem}[section]
\newtheorem{lem}[thm]{Lemma}
\newtheorem{prop}[thm]{Proposition}
\newtheorem{cor}[thm]{Corollary}
\newtheorem{assump}[thm]{Assumption}

\theoremstyle{definition}
\newtheorem{defn}[thm]{Definition}
\newtheorem{ex}[thm]{Example}
\newtheorem{rmk}[thm]{Remark}

\numberwithin{equation}{section}

\title{Finite Convergence of the Moment-SOS Hierarchy on the Product of Spheres}
\author{Sami Halaseh\footnote{LAAS-CNRS, 7 av du Colonel Roche, 31031 Toulouse Cedex 4, France.} \footnote{Fachbereich Mathematik und Statistik, Universit\"at Konstanz, Germany.} \and   Victor Magron\footnotemark[1] \footnote{Toulouse Mathematics Institute,  118 route de Narbonne, 31062 Toulouse Cedex 9, France.} \and Mateusz Skomra\footnotemark[1]}

\begin{document}
\maketitle

\begin{abstract}
We study the polynomial optimization problem of minimizing a multihomogeneous polynomial over the product of spheres. This polynomial optimization problem models the tensor optimization problem of finding the best rank one approximation of an arbitrary tensor. We show that the moment-SOS hierarchy has finite convergence in this case, for a generic multihomogeneous objective function. To show finite convergence of the hierarchy, we use a result of Huang et al. [SIAM J.Optim. 34(4) (2024), pp 3399-3428], which relies on local optimality conditions. To prove that the local optimality conditions hold generically, we use techniques from differential geometry and Morse theory. This work generalizes the main result of Huang [Optim. Lett. 17(5) (2023), pp 1263-1270], which shows finite convergence for the case of a homogeneous polynomial over a single sphere.
    

\end{abstract}
	
\section{Introduction}\label{intro}
	
A general polynomial optimization problem has the form:
\begin{equation}\label{POP}
	\begin{aligned}
		f_{\min} = \;&\inf  \; \; f (x) \\
		&\st \; h_i(x) = 0, \;\;  i \in [m],\\
		&\qquad \; g_j(x)\geq 0, \; \; j \in [\bar{m}],
	\end{aligned}
	\tag{POP}
\end{equation}
where $f,g_1,\dots, g_{\bar{m}},$ $h_1, \dots ,h_m$ are multivariate polynomials with real coefficients. A common approach to solving \eqref{POP} is to use the moment-SOS hierarchy, which a sequence of semidefinite relaxations. In this paper, we prove the finite convergence of the moment-SOS hierarchy for the polynomial optimization problem of minimizing a multihomogeneous objective polynomial over the product of spheres. We introduce the notation needed to state the problem. Consider the tuple $(n_1,\dots,n_m)\in\N^m$, and set $N= \sum_{i =1}^m n_i$. We assume that $n_i\geq 2$ for all $i$ in $[m]$. Set $\bm{x}_i = (x_{i1},\dots x_{in_i})$, for each $i\in [m]$, and $\bm{x} = (\bm{x}_1 , \dots , \bm{x}_m) \in  \R^N$. With this we have the following shorthand for a real polynomial ring,
\begin{equation}\label{vectorpolyring}
	\R[\bm{x}] = \R[\bm{x}_1,\dots, \bm{x}_m] = \R[x_{11},\dots x_{1n_1}, x_{21},\dots x_{2n_2} \dots, x_{m1},\dots x_{mn_m}].
\end{equation}

For a monomial in $\R[\bm{x}]$ and $i\in [m]$, we define the $i^{th}$ \textit{degree}, to be the sum over $j\in [n_i]$ of the exponents of the variables $x_{ij}$. We say a polynomial $f\in \R[\bm{x}]$ is \textit{multihomogeneous} of \textit{multidegree} $(d_1,\dots,d_m)$, if the $i^{th}$ degree of every monomial of $f$ is equal to $d_i$, for each $i\in [m]$. We set $\R[\bm{x}]_{=(d_1,\dots,d_m)}$ as the set of all multihomogeneous polynomials of multidegree $(d_1,\dots,d_m)$. In particular, a multihomogeneous polynomial of multidegree $(d_1,\dots,d_m)$ is homogeneous of degree $\sum_i d_i$.
	
With $\sph^{n_i-1} = \{\bm{x}_i\in \R^{n_i}\;|\;  1 - \|\bm{x}_i\|^2 = 0\}$ as the unit sphere in $\R^{n_i}$, we refer to $\sph^{n_1-1}\times \dots \times \sph^{n_m-1} \subset \R^N$ as the \textit{product of spheres}. For $f\in \R[\bm{x}]$ multihomogeneous, we are interested in the polynomial optimization problem of finding the minimum of $f$ on the product of spheres. The moment-SOS hierarchy gives a nondecreasing sequence of lower bounds on the minimum of $f$. By Putinar's Positivstellensatz, this sequence will converge to the minimum of $f$ asymptotically. However, in practice, it is preferable to know that the moment-SOS hierarchy converges in finitely many steps. This leads to our main result, which makes use of the finite convergence result in \cite{huangnie_newflatness_2024}.

\begin{thm}\label{main_result}
Let $f$ be a generic real multihomogeneous polynomial. For the following polynomial optimization problem of minimizing $f$ on the product of spheres
\begin{equation}\label{POPmultisphere}
	\begin{aligned}
	f_{\min} = \;  &\inf  \; \; f (\bm{x})\\
	& \st \; \bm{x} = (\bm{x}_1,\dots, \bm{x}_m) \in \mathbb{S}^{n_1-1}\times \dots \times \mathbb{S}^{n_m-1},
\end{aligned}
\end{equation}
the moment-SOS hierarchy has finite convergence. Furthermore, the optimality certificate of flat truncation for the moment hierarchy holds for a sufficiently high order of the hierarchy.
\end{thm}

\begin{rmk}
The statement above also holds for any larger space of polynomials, such as generic homogeneous polynomials or generic nonhomogeneous polynomials, with the same proof.
\end{rmk}

Genericity in the statement above refers to the coefficient space of the objective function. In other words, for a multihomogeneous polynomial $f$ with coefficients outside a Lebesgue measure zero subset of the space of coefficients, the moment-SOS hierarchy has finite convergence for \eqref{POPmultisphere}.

The multihomogeneous structure of the objective function allows the polynomial optimization problem \eqref{POPmultisphere} to model the tensor optimization problem of finding the best rank one approximation of an arbitrary tensor, as studied in \cite{nie_rankone_2014}. We describe this application in Subsection \ref{tensorapp}.

\subsection*{Related Works}

The case $m=1$ in Theorem \ref{main_result} was proven by Huang \cite[Corollary 3.5]{huang_homosphere_2023}. Huang's work relies on the finite convergence result of \cite{nie_localopt_2014}. Instead, we use a result from \cite{huangnie_newflatness_2024}, which improves that of \cite{nie_localopt_2014}. Both results require that certain local optimality conditions hold at all the global minimizers of our polynomial optimization problem. Huang shows that the local optimality conditions hold generically by using projective elimination theory to exploit the algebraic structure of the problem. In the case of \eqref{POPmultisphere}, where a multihomogeneous objective function is considered, this technique fails. The structure of a multihomogeneous polynomial causes `too much' to be eliminated, and the information is lost. However, Huang's argument is likely to work for an objective function which is just homogeneous (rather than multihomogeneous) on the product of spheres.

The work of Lee and Ph{\d{a}}m in \cite{lee_genericcpt_2016,lee_generic_2017} is also closely related. The authors consider a polynomial optimization problem with some assumptions on the objective and constraining polynomials. They show, using a variant of Sard's theorem, that if a generic linear perturbation is applied to the objective and constraining polynomials, the minimizers of the perturbed polynomial optimization problem will satisfy certain local optimality conditions. While our techniques are similar, there are two important points of divergence. First, we do not consider perturbations of the objective function, which would potentially alter its multihomogeneous structure. Rather, we show that outside some negligible subset of the input space, any objective function satisfies the required conditions. Second, by using techniques from Morse theory, we weaken the assumptions needed to apply Sard's theorem. These points are especially crucial when the structure of the objective function is rigid. This is the case in \eqref{POPmultisphere}, where we must preserve multihomogeneity.

In Section \ref{inequalities_sec}, we use the differential geometric notion of transversality to prove the genericity of local optimality conditions for the minimizers of the problem \eqref{POPmultisphere} with additional inequality constraints. This was motivated by the use of transversality in \cite{springarn_optcondi_1982}.

The work of the second author in \cite{magron_convgrate_2025}, studies the rate of convergence of the moment-SOS hierarchy when applied to a polynomial optimization problem on the product of spheres. 
The resulting quadratic convergence rate is obtained by extending the polynomial kernel technique used for a single sphere in \cite{fang2021sum}, and for other feasible sets in \cite{slot2022sum}; see \cite{laurent2024overview} for a recent survey of convergence rates for hierarchies in polynomial optimization. 
%When coupled with Theorem \ref{main_result}, this work shows that the moment-SOS hierarchy is a promising technique for problems which can be modeled by the polynomial optimization problem \eqref{POPmultisphere}. 
When coupled with with Theorem \ref{main_result}, the results of \cite{fang2021sum,magron_convgrate_2025} establish that the moment-SOS hierarchy  provides a constructive and computationally tractable approach for polynomial optimization problems formulated as in \eqref{POPmultisphere}.
To strengthen these results, further work can be done to find an upper bound on the finite convergence order, see Section \ref{further_work_sec}.

The problem \eqref{POPmultisphere} models the tensor optimization problem of finding the best rank one approximation of a given tensor, see Subsection \ref{tensorapp} for details. This tensor optimization problem has received attention from other fields. For example, in \cite{friedland_rankonelinalg_2013}, linear algebra techniques are used, while in \cite{chen_rankoneriemann_2025}, techniques from Riemannian geometry are used. The formulation and approach in this paper can be found in \cite{nie_rankone_2014}. There, extensive numerical results are presented. This situates our work in a recent line of research studying the applications of polynomial optimization and semidefinite programming to tensor optimization problems. For further reference, see \cite{nie_tensornuclear_2017,nie_cubic_2025,nie_robusttensor_2025,huangnie_tensor_2024}. 
An important problem arising from quantum information theory is the one of testing the separability of quantum states. 
This separability test can be practically achieved via the DPS hierarchy of semidefinite programs from \cite{doherty2004complete}. 
%\cite{harrow_sdpfordps_2017}
We comment further on connections to the DPS hierarchy in Section \ref{further_work_sec}.

\subsection*{Outline}

We briefly outline the paper. In Section \ref{preliminaries}, we introduce the necessary preliminaries: local optimality conditions from classical nonlinear optimization, the moment-SOS hierarchy, and the necessary background for our main application to tensor optimization. In Section \ref{morse_sec}, we develop the tools needed to prove our Theorem \ref{main_result}. These tools come from differential geometry and Morse theory. An effort has been made to make these notions accessible to the nonexpert. Section \ref{finiteconvergence_sec} is dedicated to the proof of Theorem \ref{main_result}. In Section \ref{inequalities_sec}, we extend our main result to include inequality constraints. In that case, we show finite convergence of the moment-SOS hierarchy using the notion of parametric transversality from differential geometry. In Section \ref{further_work_sec}, we close by describing some directions for possible further work.

\subsection*{Acknowledgements}

This work has been supported by European Union’s HORIZON–MSCA-2023-DN-JD programme under the Horizon Europe (HORIZON) Marie Skłodowska-Curie Actions, grant agreement 101120296 (TENORS).

\section{Preliminaries}\label{preliminaries}

\subsection{Local Optimality Conditions}\label{local_opt_cond_section}
	
We begin by reviewing local optimality conditions from classical nonlinear optimization. A classic reference is \cite{bertsekas_nonlinear_2016}. As stated in the introduction, we have polynomials $f,g_1,\dots,g_{\bar{m}}, h_1,\dots,h_m\in\R [x]$ defining an optimization problem \eqref{POP}. A feasible point $u$ is called a \textit{local minimum}, if there exists $\varepsilon > 0$ such that 
\[\text{for all}\; x \in B_\epsilon(u)\cap \manif \;\; f(u)\leq f(x),\]
where $B_\varepsilon(u)$ denotes the ball of radius $\varepsilon$ centered at $u$ and $\manif$ is the feasible semialgebraic set of \eqref{POP}. When the inequality above is strict for all $x \neq u$, we say $u$ is a \textit{strict local minimum}. In the following, we suppose $u$ is a local minimum of \eqref{POP}. Define the set of \textit{active constraints at} $u$ as $J(u) = \{j \in [\bar{m}] \;|\; g_j(u)=0\}$. So for $j\in [\bar{m}] \setminus J(u)$ we have that $g_j(u)>0$ is an \textit{inactive constraint}. We say the \textit{linear independence constraint qualification condition} \eqref{LIC} is satisfied at $u$ if the set of gradients
\begin{equation}\label{LIC}
	\{\nabla_x h_i(u)\}_{i\in[m]} \cup \{\nabla_x g_j(u)\}_{j\in J(u)}
	\tag{LIC}
\end{equation}
is linearly independent. This is also known as \textit{regularity}. 
For $u$ satisfying \eqref{LIC}, there exist real scalars $\{\lambda_i\}_{i\in[m]} \cup \{\mu_j\}_{j\in[\bar{m}]}$ such that the following conditions, called \textit{first order optimality condition} \eqref{FOOC} and the \textit{complementary condition} \eqref{CC}, hold:
\begin{equation}\label{FOOC}
	\nabla_x f (u) \; = \; \sum_{j=1}^{\bar{m}} \mu_j \nabla_x g_j(u)\; + \;\sum_{i=1}^m \lambda_i \nabla_x h_i(u)
	\tag{FOOC}
\end{equation}
\begin{equation}\label{CC}
\begin{aligned}
\mu_j \geq 0, \;\; &\text{for all} \;\; j \in J(u),\\
\mu_j=0, \;\; &\text{for all} \;\; j \in [\bar{m}]\setminus J(u).
\end{aligned}
\tag{CC}
\end{equation}
These two conditions together are known as the \textit{Karush Kuhn Tucker} (KKT) conditions. The scalars $\{\lambda_i\}_{i\in[m]} \cup \{\mu_j\}_{j\in[\bar{m}]}$ are known as the \textit{Lagrange multipliers} of $u$. If a feasible point $p$ (not necessarily a local minimum) satisfies the (KKT) conditions, we say it is \textit{critical point}. If, for a local minimum $u$, we have
\begin{equation}\label{SCC}
\begin{aligned}
\mu_j > 0, \;\; &\text{for all} \;\; j \in J(u)\\
\mu_j=0, \;\; &\text{for all} \;\; j \in [\bar{m}]\setminus J(u)
\end{aligned}
\tag{SCC}
\end{equation}
we say $u$ satisfies the \textit{strict complementary condition} \eqref{SCC}. We define the \textit{Lagrangian} of \eqref{POP} as 
\begin{equation}\label{lagrangian_function}
	\mathcal{L}(x,\lambda, \mu) = f(x) - \sum_{j=1}^{\bar{m}} \mu_j g_j(x) - \sum_{i=1}^m \lambda_i h_i(x).
\end{equation}
This allows us to write the condition \eqref{FOOC} more concisely as,
\[\nabla_x \mathcal{L} (u,\lambda,\mu) = 0.\]
Also under the \eqref{LIC} assumption on $u$, we have the \textit{second order necessary condition} \eqref{SONC}, which uses the Hessian of the Lagrangian,
\begin{equation}\label{SONC}
	v^T \left(\nabla^2_x \mathcal{L} (u)\right)v \geq 0 \quad\text{for all}\quad v\in \left(\,\bigcap_{i=1}^m \nabla_x h_i(u)^\perp \right) \;\bigcap\; \left(\,\bigcap_{j\in J(u)} \nabla_x g_j(u)^\perp\right),
	\tag{SONC}
\end{equation}
where $\nabla_x h_i(u)^\perp = \{v \in \R^n \;|\; h_i(u)^Tv = 0\}.$ When the inequality above is strict for $v\neq0$, we say $u$ satisfies the \textit{second order sufficiency condition} \eqref{SOSC} 
\begin{equation}\label{SOSC}
	v^T \left(\nabla^2_x \mathcal{L} (u)\right)v > 0 \quad\text{for all}\quad v\in \left(\,\bigcap_{i=1}^m\nabla_x h_i(u)^\perp\right) \; \bigcap \; \left(\,\bigcap_{j\in J(u)} \nabla_x g_j(u)^\perp\right) \setminus \{0\}.
	\tag{SOSC}
\end{equation}
The strict complementary condition and the second order sufficiency condition are sufficient to prove that $u$ is a local minimum. In summary, we have the following proposition.
\begin{prop}[{\cite[Sections 3.1 \& 3.2]{bertsekas_nonlinear_2016}}]\label{local_min_conditions}
	If $u$ is a local minimum of \eqref{POP} satisfying \eqref{LIC}, then it also satisfies the conditions \eqref{FOOC}, \eqref{CC}, and \eqref{SONC}. 
	
	Alternatively, if a feasible point $u$ satisfies the conditions \eqref{FOOC}, \eqref{SCC}, and \eqref{SOSC}, then $u$ is a strict local minimum.
\end{prop}

\subsection{The Moment-SOS Hierarchy}\label{mom/sos_section}
	
In this subsection, we introduce the moment-SOS hierarchy and discuss its convergence and certification. References for the moment-SOS hierarchy include \cite{lasserre_introtopop_2015,laurent_survey_2009,nie_momentbook_2023}. References for real algebraic geometry include \cite{powers_certificates_2021,scheiderer_course_2024,schmudgen_moment_2017}.
	
In the general polynomial optimization problem \eqref{POP}, the feasible semialgebraic set is given as
\begin{equation}\label{feasible_semialg_set}
\manif = \{x\in\R^n \; | \; h_i(x) = 0 \;\text{for all}\; i\in [m]\} \;\cap \;\{x\in\R^n \; | \; g_j(x)\geq 0 \; \text{for all}\; j \in [\bar{m}]\}.
\end{equation}

Let $\Sigma[x]\subset \R[x] = \R[x_1,\dots,x_n]$ be the set of polynomials which can be written as a sum of squares. We define the \textit{ideal} generated by the polynomials $h_i$, and the \textit{quadratic module} generated by the polynomials $g_j$ as 
\[(h) = (h_1,\dots,h_m) = \left\{\sum_{i=1}^{m} p_i h_i \;|\; p_i\in\R [x] \;\text{for}\; i\in[m] \right\},\]
\[\QM(g) =\QM(g_1,\dots,g_{\bar{m}}) = \Bigg\{\sigma_0 +\sum_{j=1}^{\bar{m}}\sigma_jg_j \;|\; \sigma_j \in \Sigma[x] \;\text{for}\; j\in[\bar{m}]\cup \{0\}\Bigg\}.
\]
If the problem has no inequality constrains ($\bar{m}=0$), then we set $\QM(g) = \Sigma[x]$. If a polynomial $p$ is in the set $(h)+\QM(g)$, then is it clearly nonnegative on the feasible set $\manif$. We use this to write the \textit{SOS relaxation} of the problem \eqref{POP},
\begin{equation}\label{SOS}
	\begin{aligned}
		f_{\sos}=\; &\sup \; \; \lambda \\
		&\st\;  f-\lambda \in (h) + \QM(g).
	\end{aligned}
	\tag{SOS}
\end{equation}
Here $\lambda$ gives a lower bound approximation on $f_{\min}$. Truncating the degree of the polynomials turns \eqref{SOS} into a finite dimensional program,
\begin{equation}\label{SOSk}
	\begin{aligned}
		f_{\sos}^k = \; &\sup \; \; \lambda \\
		&\st  \;  f-\lambda \in (h)_{2k} + \QM(g)_{2k},
	\end{aligned}
	\tag{SOS, $k$}
\end{equation}
where
\[(h)_{2k} = \left\{\sum_{i=1}^{m} p_ih_i \;|\; p_i \in \R [x] \; \text{for}\; i \in [m], \; \deg(p_ih_i)\leq 2k\right\},\]
\[\QM(g)_{2k} = \Bigg\{\sigma_0 + \sum_{j=1}^{\bar{m}} \sigma_j g_j \;|\; \sigma_j \in \Sigma[x]\; \text{for}\; j\in [\bar{m}] \cup \{0\},\; \deg(\sigma_j g_j) \leq 2k\Bigg\}.\]
Thus, we have that $(h)_{2k} + \QM(g)_{2k}$ is a subset of the finite dimensional vector space of polynomials with degree less than or equal to $2k$, denoted by $\R [x]_{2k}$. 
The use of Gram representation for SOS polynomials turns \eqref{SOSk} into a semidefinite program;  see \cite[Chp 3]{powers_certificates_2021} for details. 

\begin{defn}\label{archQM}
	Given polynomials $h_1,\dots h_m,$ $g_1,\dots, g_{\bar{m}} \in \R [x]$, we say that the sum $(h) + \QM(g)$ is \textit{archimedean}, if there exists $R >0$, such that $R-\sum_{i=1}^{n} x_i^2 \in (h) + \QM(g)$.
\end{defn}
As an example, in the problem \eqref{POPmultisphere}, let $s_1(\bm{x}),\dots,s_m(\bm{x})$ be the polynomials giving the product of spheres constraint, i.e., 
\[
s_i(\bm{x}) = 1 - \|\bm{x}_i\|^2, \;\; \text{for each}\;\; i \in [m].
\]
Then, the sum $(h) + \QM(g) = (s) + \Sigma[x]$ is archimedean as it contains the ball constraint 
\[\sum_{i=1}^m s_i(\bm{x}) = \sum_{i=1}^m (1 - \|\bm{x}_i\|^2)= m - \|\bm{x}\|^2.\]
We now introduce the dual side of the hierarchy. Consider the following reformulation of \eqref{POP},
\begin{equation}\label{prob}
	\begin{aligned}
	f_{\min} = \;&\inf  \; \; \int_{\manif} f \; \di \mu \\
		&\st \; \mu\;  \text{is a probability measure supported on}\; \manif.
	\end{aligned}
\end{equation}
Define the $\alpha$ \textit{moment} of the measure $\mu$ as $y_\alpha = \int_{\manif} x^\alpha \di \mu$. This gives an infinite sequence $y=(y_\alpha)_{\alpha\in \N^n}$, which we call the \textit{moment sequence of} $\mu$.

\begin{defn}\label{moment_matrix}
Fix a sequence $y\in \R^{\N^n}$ and polynomial $g\in \R [x]$, with $g=\sum_{\alpha\in \N^n}g_\alpha x^\alpha$. We define the \textit{pseudo-moment matrix} associated with $y$ as
$$M(y)_{\alpha,\beta} = y_{\alpha+\beta},\;\; \alpha,\beta \in \N^n.$$
If $y$ comes from a measure, we say $M(y)$ is the \textit{moment matrix} of $y$. 
Let $\vecc(g)$ be the vector of coefficients of the polynomial $g$. We define a new sequence $gy \in \R^{\N^n}$, with $gy= M(y)\vecc(g),$ and $ (gy)_\alpha = \sum_{\delta\in \N^n} g_\delta y_{\alpha +\delta}$, which is a finite sum, since $g$ is a polynomial. We associate the following matrix with the sequence $gy$, 
\begin{equation}\label{shiftedsequence}
M(gy)_{\alpha,\beta} = (gy)_{\alpha+\beta} = \sum_{\delta\in \N^n} g_\delta y_{\alpha +\beta+\delta}.
\end{equation}
The matrix $M(gy)$ is called the \textit{localizing matrix} associated with $g$ and $y$. 
\end{defn}

The matrices defined above are infinite. As we did with the ideal and quadratic module, we define a truncation.

\begin{defn}
For a multi-index $\alpha \in \N$, we set $|\alpha| = \alpha_1 + \dots + \alpha_n$. Given $k \in \N$ and $y\in \R^{\N^n}$, we define the \textit{truncated moment matrix} of $y$ as $M_k(y)_{\alpha,\beta} = (y)_{\alpha+\beta}$, for all $\alpha, \beta \in \N^n$ with $|\alpha|, |\beta| \leq k$. Similarly, given $g\in \R [x]$, the \textit{truncated localizing matrix} of the sequence $gy$ is defined by $M_k(gy)_{\alpha,\beta} = (gy)_{\alpha+\beta}$, for all $\alpha, \beta \in \N^n$ with $|\alpha|, |\beta| \leq k$. 
\end{defn}
The following theorem, known as the dual version of Putinar's Positivstellensatz, gives us the sufficient conditions we need for an arbitrary sequence $y\in \R^{\N^n}$ to be the moment sequence of a measure. For a matrix $A$, we write $A\succeq0$ to denote that $A$ is positive semidefinite. We say the infinite pseudo-moment matrix $M(y)$ is positive semidefinite,
if the truncated pseudo-moment matrix $M_k(y)$ is positive semidefinite, for all $k \in \N$.

\begin{thm}[{\cite[Theorem 2.44]{lasserre_introtopop_2015}}]\label{Putinarmom}
Let $\manif$ be the semialgebraic set given in \eqref{feasible_semialg_set}, and suppose its defining polynomials are such that $(h)+\QM(g)$ is archimedean. Consider a sequence $y\in\R^{\N^n}$. There exists a measure $\mu$ with support contained in $\manif$, such that $y_\alpha = \int x^\alpha \di \mu$ for all $\alpha \in \N^n$, if and only if, the following conditions hold,
\[
M(y) \succeq 0, \; M(g_jy) \succeq 0, \;\text{and}\; M(h_iy)=0, \;\text{for all} \; j\in [\bar{m}] \; \text{and} \; i \in [m].
\]
\end{thm}

The above theorem allows us to formulate the \textit{moment relaxation} \eqref{mom} of the polynomial optimization problem \eqref{POP}, as
\begin{equation}\label{mom}
\begin{aligned}
	f_{\mom} = \;&\inf &&f^Ty\\
	&\st && M(y)\succeq 0 \\
	&\;\; &&M(g_jy)\succeq 0,\; j\in [\bar{m}]\\
	&\;\; &&M(h_iy) = 0, \; i \in [m]\\
	&\;\;  &&y_0 =1, \; y \in \R^{\N^n}.
\end{aligned}
\tag{mom}
\end{equation}
The first index of $y$ equalling one above, corresponds to the fact that we considered a probability measure in \eqref{prob}. We set
\begin{equation}\label{degrees}
d_{g_j} = \left\lceil \frac{\deg (g_j)}{2}\right\rceil \quad \text{and} \quad
d_{h_i} = \left\lceil \frac{\deg (h_i)}{2}\right\rceil.
\end{equation}
We truncate the relaxation to the following finite dimensional problem,
\begin{equation}\label{momk}
	\begin{aligned}
		f_{\mom}^k = \;&\inf &&f^Ty\\
		&\st && M_k(y)\succeq 0 \\
		&\;\; &&M_{k-d_{g_j}}(g_jy)\succeq 0,\; j\in [\bar{m}]\\
		&\;\; &&M_{k-d_{h_i}}(h_iy) = 0, \; i\in[m]\\
		&\;\;  &&y_0 =1, \; y \in \R^{\N^n_{2k}},
	\end{aligned}
	\tag{mom, $k$}
\end{equation}
which is the dual semidefinite program to \eqref{SOSk}.
The discussion above leads to the following definition.
\begin{defn}\label{mom/sos}
We define the \textit{moment-SOS hierarchy} as the two sequences of semidefinite programs \eqref{SOSk} and \eqref{momk} for $k\in \N$.
\end{defn}

\begin{defn}\label{finiteconvergencedef}
We say the moment-SOS hierarchy \textit{has finite convergence}, or \textit{converges at a finite order}, if for some $k_0\in \N$, we have 
\[f_{\sos}^{k_0}  = f_{\mom}^{k_0} = f_{\min}, \;\; \text{for all} \;\; k\geq k_0.\]
\end{defn}

By Putinar's Positivstellensatz (\cite[Theorem 7.3]{powers_certificates_2021} or \cite[Theorem 5.3.1]{scheiderer_course_2024}), when $(h) + \QM(g)$ is archimedean, the sequence $\{f_{\sos}^k\}$ converges to $f_{\min}$ as $k$ tends to infinity. See \cite[Theorem 6.2]{lasserre_globalopt_2000}, \cite[Theorem 6.8]{laurent_survey_2009} or \cite[Theorem 5.2.2]{nie_momentbook_2023} for details. However, in practice, it is preferable to know that the moment-SOS hierarchy converges in finitely many steps. If we have a result guaranteeing finite convergence, we would still need a certificate to know at which order convergence has occurred. Such an optimality certificate will come from the moment side of the hierarchy. With $d_f =\left\lceil \frac{\deg (f)}{2}\right\rceil$, set
\begin{equation}
d_{\manif} = \max\big\{1, d_{g_j},d_{h_i}\; |\; j\in[\bar{m}], \; i\in[m]\big\} \quad \text{and} \quad \bar{d} = \max\{d_f,d_{\manif}\}.
\end{equation}

\begin{defn}
We define the \textit{flat truncation condition} as, the existence of an integer $t \in [\bar{d},k]$, such that
\begin{equation}\label{flat_trunc_cond}
\rank M_{t-d_{\manif}} (y) = \rank M_{t}(y).
\end{equation}
where $k$ is as in \eqref{momk} and the notation of \eqref{degrees} is used.
\end{defn}

\begin{thm}[{\cite[Theorem 6.18]{laurent_survey_2009}}]\label{flat_trunc_thm}
Suppose $y$ is an optimizer of \eqref{momk}, which satisfies the flat truncation condition \eqref{flat_trunc_cond} for some integer $t$. Then, we have finite convergence of the moment hierarchy at order $k$, that is $f_{\mom}^k = f_{\min}$, and $y$ has an associated atomic measure of the form
\[
\mu = \sum_{i=1}^r c_i \, \delta_{x(i)}, \quad \text{with}\;\; c_i\in \R \;\; \text{and} \;\; x(i) \in \manif \;\; \text{for all} \;\; i \in [r],
\]
where $\delta_{(\cdot)}$ denotes the Dirac measure, $r=\rank M_t(y)$ and $\manif$ is the feasible set described in \eqref{feasible_semialg_set}.
\end{thm}

The points $x(i)$ in the statement above are precisely the minimizers of \eqref{POP}. There is a numerical linear algebra based algorithm to extract these minimizers, see \cite{henrion_detecting_2005} or \cite[Section 6.7]{laurent_survey_2009}.

We now present the result that will be one of the main tools to proving our result, Theorem \ref{main_result}. 

\begin{thm}[{\cite[Theorem 3.2]{huangnie_newflatness_2024}}]\label{KKTfiniteconvergence}
Consider the polynomial optimization problem \eqref{POP}, and suppose $(h)+\QM(g)$ is archimedean. If the linear independence constraint qualification condition \eqref{LIC}, the strict complementary condition \eqref{SCC}, and the second order sufficiency condition \eqref{SOSC} hold at every global minimizer of \eqref{POP}, then the moment-SOS hierarchy applied to \eqref{POP} has finite convergence. 

Furthermore, every solution of the moment relaxation \eqref{momk} satisfies the flat truncation condition for a sufficiently high order of the hierarchy $k$.
\end{thm}

The theorem above is a strengthening of the result \cite[Theorem 1.1]{nie_localopt_2014}, both of which build off work by Marshall, see \cite[Chapter 9.5]{marshall_positive_2008}.

In \cite{huang_homosphere_2023}, the polynomial optimization problem of a generic homogeneous polynomial as the objective function on the unit sphere is studied. There, \cite[Theorem 1.1]{nie_localopt_2014} is used to show finite convergence of the moment-SOS hierarchy.

\begin{thm}[{\cite[Corollary 3.5]{huang_homosphere_2023}}]\label{sphere_finite}
Let $f$ be a generic homogeneous polynomial. For the following polynomial optimization problem of minimizing $f$ on the unit sphere
\begin{equation}\label{sphere}
	\begin{aligned}
		f_{\min} = \;&\inf  \; \; f (x)\\
		&\st \; x \in \mathbb{S}^{n-1},
	\end{aligned}
\end{equation}
the moment-SOS has finite convergence.
\end{thm}
	
Note that problem \eqref{POPmultisphere} on the product of spheres generalizes problem \eqref{sphere}. Thus, our aim is to generalize Theorem \ref{sphere_finite}, by showing that the moment-SOS hierarchy applied to the polynomial optimization problem \eqref{POPmultisphere} has finite convergence. For this, we use Theorem \ref{KKTfiniteconvergence}.
	
\subsection{Application to Tensors}\label{tensorapp}

In this subsection, we use the polynomial optimization problem \eqref{POPmultisphere} to model a problem in tensor optimization, as discussed in \cite{nie_rankone_2014}. We introduce the necessary background on tensor theory, references of which include \cite{comon_symtensors_2008,hogben_linalghandbook_2014,landsberg_tensors_2012}.

For $n_1,\dots, n_m \in \N$, we define a \textit{real tensor $\mathcal{A}$ of order m}, as an element of the space $\R^{n_1\times \dots \times n_m}$. While tensors may have entries from any field, we consider tensors with real entries and remark about tensors with complex entries below at the end of the section.

For two tensors $\mathcal{A}, \mathcal{B} \in \R^{n_1\times \dots \times n_m}$, the \textit{Hilbert Schmidt inner product} is given by
\[\langle \mathcal{A},\mathcal{B}\rangle = \sum_{j_1=1}^{n_1} \; \dots \sum_{j_m=1}^{n_m} \; \mathcal{A}_{j_1\dots j_m}\; \mathcal{B}_{j_1\dots j_m}.\]
The inner product is symmetric, i.e., $\langle \mathcal{A} ,\mathcal{B} \rangle = \langle \mathcal{B}, \mathcal{A}\rangle$. We define the \textit{Hilbert Schmidt norm} as,
\[\|\mathcal{A}\|^2 = \langle\mathcal{A},\mathcal{A}\rangle = \sum_{j_1=1}^{n_1} \; \dots \sum_{j_m=1}^{n_m} \; |\mathcal{A}_{j_1\dots j_m}|^2.\]

\begin{defn}
We say a tensor $\rkone \in \R^{n_1\times \dots \times n_m}$ is \textit{rank one}, if it can be written as the outer product of unit vectors $\bm{u}_i \in \R^{n_i}$ times a scalar $a\in \R \setminus \{0\}$, i.e.,
\begin{equation}\label{rankonetensor}
	\rkone = a \bm{u}_1 \otimes \dots \otimes \bm{u}_m, \; a\in \R \setminus \{0\}, \; \|\bm{u}_i\| = 1 \;\text{for all}\; i \in [m].
\end{equation}
\end{defn}

We can associate a polynomial to a tensor $\mathcal{A} \in \R^{n_1\times \dots \times n_m}$ through the inner product, 
\begin{equation}\label{tensor_poly}
    \mathcal{A}(\bm{x}) = \;\langle \mathcal{A}, \bm{x}_1 \otimes \dots \otimes \bm{x}_m \rangle \; = \;\sum_{j_1=1}^{n_1} \dots \sum_{j_m =1}^{n_m}\; \mathcal{A}_{j_1\dots j_m} \; x_{1 j_1} \dots x_{m j_m}.
\end{equation}
Note that, the polynomial above is multilinear in the vectors $\bm{x}_1,\dots , \bm{x}_m$, i.e., with the notation in the introduction $\mathcal{A}(\bm{x}) \in \R[\bm{x}]_{=(1,\dots,1)}$.
We give an example of this notation.
\begin{ex}
	Consider the following real tensor $\mathcal{A} \in \R^{2\times 3 \times 2}$
	\begin{equation}\label{arbtensorexample}
	\mathcal{A} = \left.
	\left(\begin{matrix}
		2 &5&3\\
		-1&3&0
	\end{matrix}\; \; \right\vert \;\;
	\begin{matrix}
		-8&4&1\\
		7&-6&-2
	\end{matrix}\right),
\end{equation} 
where we write the tensor slice by slice over the third index, i.e., the entries in first block on the left have their third index equalling one. We set 
\[\bm{x}_1 = \left(\begin{matrix} x_{11} \\ x_{12}\end{matrix}\right), \qquad \bm{x}_2 = \left(\begin{matrix} x_{21} \\ x_{22} \\ x_{23}\end{matrix}\right), \qquad \bm{x}_3 = \left(\begin{matrix} x_{31} \\ x_{32}\end{matrix}\right).\]
Then, from \eqref{tensor_poly}, we have the following multilinear polynomial,
\[
\begin{aligned}
\langle \mathcal{A},\bm{x}_1 \otimes \bm{x}_2 \otimes \bm{x}_3 \rangle &= \sum_{j_1 = 1 }^{2} \; \sum_{j_2=1}^{3} \; \sum_{j_3=1}^{2} \mathcal{A}_{j_1 j_2 j_3} \; x_{1j_1} \; x_{2j_2} \; x_{3j_3} \\
&= 2 x_{11} x_{21} x_{31} - 8x_{11} x_{21} x_{32} + 5 x_{11} x_{22} x_{31} + 4 x_{11} x_{22} x_{32} \\
&\;\;\; + 3 x_{11} x_{23} x_{31} + x_{11} x_{23} x_{32} -  x_{12} x_{21} x_{31} +7 x_{12} x_{21} x_{32}\\
&\;\;\; + 3 x_{12} x_{22} x_{31} -6 x_{12} x_{22} x_{32} -2 x_{12} x_{23} x_{32}.
\end{aligned}
\]
\end{ex} 

Given a tensor $\mathcal{A} \in \R^{n_1\times \dots \times n_m}$, we consider the problem of finding the best rank one approximation of $\mathcal{A}$. Setting $\rkonex = a \,\bm{x}_1 \otimes \dots \otimes \bm{x}_m$ as an indeterminate rank one tensor, we formulate the problem as,
\begin{equation}\label{bestrankone}
\begin{aligned}
& \min \;\; &&\|\mathcal{A} - \mathcal{X}\|^2 \\
&\st\; &&\rank(\mathcal{X})= 1 \\
&\; &&\mathcal{X} \in \R^{n_1\times \dots \times n_m}.
\end{aligned}
\end{equation}
We use the notions above to simplify the objective function of \eqref{bestrankone}:
\[ \|\mathcal{A} - \mathcal{X}\|^2 = \langle \mathcal{A} - \mathcal{X}, \mathcal{A} - \mathcal{X}\rangle = \|\mathcal{A}\|^2 - 2 a \langle \mathcal{A},\bm{x}_1 \otimes \dots \otimes \bm{x}_m\rangle + a^2.\]
Treating $a$ as a variable, we have a quadratic univariate polynomial, which attains its minimum at $a = \langle \mathcal{A},\bm{x}_1 \otimes \dots \otimes \bm{x}_m\rangle$. This value of $a$ gives the minimum value of the univariate quadratic as $\|\mathcal{A}\|^2 - a^2$. This means that to minimize the value of the problem \eqref{bestrankone}, we need to maximize the value of $|a|$. With this, we can eliminate $a$ from \eqref{bestrankone} to get the equivalent problem,
\begin{equation}\label{bestrankonesimple}
\begin{aligned}
&\max \;\; | \langle \mathcal{A},\bm{x}_1 \otimes \dots \otimes \bm{x}_m\rangle| \\
& \st \; \bm{x} = (\bm{x}_1,\dots, \bm{x}_m) \in \mathbb{S}^{n_1-1}\times \dots \times \mathbb{S}^{n_m-1}.
\end{aligned}
\end{equation}
By the association of multihomogeneous polynomials to tensors, the problem \eqref{bestrankonesimple} becomes a pair of polynomial optimization problems,
\begin{equation}\label{pairofPOPs}
\begin{aligned}
 	a^+ = &\max \;\; \mathcal{A}(\bm{x}) \quad \qquad \qquad\qquad a^- = \hspace{-3mm} &&\min \;\; \mathcal{A}(\bm{x}) \\
 	&\st\;\; \bm{x}_i \in \sph^{n_i-1} &&\st\;\; \bm{x}_i \in \sph^{n_i-1}\\
 	&\qquad \quad i \in [m], &&\qquad \quad i \in [m].
\end{aligned}
\end{equation}
Due to the absolute value in \eqref{bestrankonesimple}, the maximal value of the problem \eqref{bestrankonesimple} is $\max\{|a^+|,|a^-|\}$. 

Thus, the tensor optimization problem of finding the best rank one approximation of a given tensor $\mathcal{A} \in \R^{n_1\times \dots \times n_m}$ can be modeled by the polynomial optimization problem \eqref{POPmultisphere}. With this, we have the following corollary of Theorem \ref{main_result}.

\begin{cor}\label{tensor_finite_converg}
Let $\mathcal{A} \in \R^{n_1\times \dots \times n_m}$ be a generic real tensor. For the reformulation of the problem of finding the best rank one approximation of $\mathcal{A}$ in \eqref{pairofPOPs}, the moment-SOS hierarchy has finite convergence. Furthermore, the optimality certificate of flat truncation for the moment hierarchy holds for a sufficiently high order of the hierarchy.
\end{cor}

\begin{rmk}
When the tensor $\mathcal{A}$ is symmetric the associated polynomial is homogeneous, see \cite[Section 3]{comon_symtensors_2008} for definitions and details. Moreover, the best rank one approximation of a symmetric tensor can be chosen to be symmetric, see \cite[Theorem 1]{friedland_rankonesym_2013}. In that case the polynomial optimization problem consists of optimizing a homogeneous polynomial over the unit sphere. Thus, for a symmetric tensor, the statement above is a corollary of Theorem \ref{sphere_finite}. 
\end{rmk}

\begin{rmk}
Suppose $(\bm{x}^*_1,\dots, \bm{x}^*_m) \in \sph^{n_1-1} \times \dots \times \sph^{n_m-1}$ is a maximizer of \eqref{bestrankonesimple}, and set $a^* = \langle \mathcal{A},\bm{x}^*_1 \otimes \dots \otimes \bm{x}^*_m\rangle$, then the best rank one tensor approximating the tensor $\mathcal{A}$ is the tensor $\rkonex^* = a\, \bm{x}^*_1 \otimes \dots \otimes \bm{x}^*_m$. Note that, to get the best rank one approximation tensor, one needs to extract the maximizers of problem \eqref{bestrankonesimple}. This is facilitated by the guarantee of the optimality certification mentioned in Corollary \ref{tensor_finite_converg}, see Section \ref{mom/sos_section} for details.
\end{rmk}

\begin{rmk}\label{complex_sphere}
For tensors with complex entries, we can double the number of variables to obtain an equivalent real polynomial optimization formulation. For details about this procedure, see \cite[Section A.1]{gribling_complex_2022}.
\end{rmk}

\section{Morse Theory for the Working Optimizer}\label{morse_sec}
	
In this section, we introduce notions from differential geometry and Morse theory, which will be the tools to prove our main result in Section \ref{finiteconvergence_sec}. An effort was made to make this an accessible introduction to the nonexpert, so that this approach can be further applied in optimization contexts. For background on differential geometry, see \cite{lee_smoothmani_2013,tu_manifolds_2011}, with \cite[Chapter 1.7]{guillemin_difftop_2010} particularly relevant. For further reading on Morse theory, see \cite{nicolaescu_morse_2011}.
We use similar symbols to the previous sections, so the reader can anticipate their application.
	
\subsection{Basic Notions of Morse Theory}\label{morse_basics_section}

Given an open set $U \subseteq \R^n$, we say that a function $\varphi: U \to \R^m$ is \emph{smooth} if it is of class $C^{\infty}$.

\begin{defn}
A \textit{$k$ dimensional manifold} $\manif$ of $\R^n$ is a subset $\manif \subseteq \R^n$, such that for every point $p\in \manif$, there exists a neighbourhood $U\ni p$ of $\R^n$ and a smooth bijection $\varphi: U \rightarrow V\subset \R^n$ such that $\varphi^{-1}$ is smooth and
\[
\varphi(U\cap \manif) = V\cap \pi_k\left(\R^n\right) \subset \R^n.
\]
Here, $\pi_k\left(\R^n\right)= \big\{(a_1,\dots,a_k,0,\dots,0) \in \R^n\;|\; a_1,\dots,a_k \in \R\big\}$ is the projection onto the first $k$ coordinates.
\end{defn}

Note that some texts define a manifold as a standalone object, i.e., not embedded in $\R^n$. The definition above can be found under the definition of regular submanifold in, for example, \cite[Chapter 9]{tu_manifolds_2011}. 
%We define a \textit{submanifold} $\mathcal{M}$ of $\manif$, as a subset of $\manif$ which satisfies the definition of a manifold, for some dimension.

\begin{defn}\label{curve}
A \textit{curve} on a manifold $\manif$ is a smooth function $\gamma: (-a,a)\rightarrow \R^n$, with $a>0$, such that the image of $\gamma$ is fully contained in $\manif$. We denote the derivative or `velocity vector' of a curve $\gamma$, as $\dot{\gamma}: (-a,a)\rightarrow\R^n$. 
\end{defn}
This leads to the definition of the tangent space at a point in a manifold.
\begin{defn}\label{tanspace_def}
We say that a curve $\gamma$ on $\manif$ is \emph{centered at $p \in \manif$} if $\gamma(0) = p$. A vector $v \in \R^n$ is a \textit{tangent vector} at $p$ if there exists a curve $\gamma$ on $\manif$ centered at $p$ and such that $\dot{\gamma}(0) = v$. The \emph{tangent space at $p$}, denoted by $T_p \manif$, is the set of all tangent vectors at $p$,
\[
T_p \manif = \big\{ v\in \R^n \; | \; \text{there exists a curve $\gamma $ on} \; \manif \; \st \; \gamma(0)=p \; \text{and} \; \dot{\gamma}(0) = v \big\}.
\]
For $v\in T_p\manif$, we say $\gamma$ is a \textit{representing} curve for $v$ if $\gamma(0)=p$ and $\dot{\gamma}(0) = v$. In other words, $v$ is an equivalence class of curves and $\gamma$ is a representative of this class, where two curves are equivalent if their velocity vectors agree at $p$.
\end{defn}
There exists other, equivalent, definitions of the tangent space, see \cite[Chapter 8.1]{tu_manifolds_2011},\cite[Chapter 3]{lee_smoothmani_2013}, or \cite[Chapter 1.2]{guillemin_difftop_2010}. The following statement describes the tangent space as a vector subspace of $\R^n$, the ambient space of our manifold $\manif$.
\begin{prop}[{\cite[Proposition 3.10]{lee_smoothmani_2013}}]\label{tanspace_dim_prop}
Consider a manifold $\manif\subset \R^n$, and a point $p\in \manif$. The tangent space $T_p\manif$ is a subspace of $\R^n$ and has dimension equal to that of the manifold $\manif$. Thus, we have the following isomorphism of real vector spaces $T_p\manif \cong \R^{\dim(\manif)}$.
\end{prop}

\begin{defn}\label{cotanspace_def}
For a point $p\in \manif$, we define the \emph{cotangent space} at $p$, denoted by $T^*_p\manif$, to be the algebraic dual of the tangent space $T_p\manif$.
\end{defn}

Since the tangent space is a finite dimensional real vector space, it and its dual are isomorphic as real vector spaces.

The central objects of study in differential geometry are smooth maps between manifolds. We consider smooth functions $f \colon \R^n \to \R^{n'}$ such that the image by $f$ of a manifold $\manif$ is included in another manifold $\manif'\subset \R^{n'}$. Suppose that $f$ is such a function. For a point $p\in \manif$, the map $f$ induces a linear map from the tangent $T_p\manif$ to $T_{f(p)}\manif'$. This comes from postcomposing a curve on $\manif$ with $f$ and applying the chain rule, as follows,
\begin{equation}\label{pushforward}
	(f\circ \gamma)' (0) = f'(\gamma(0)) \,\gamma'(0)= 
	\left(\begin{matrix}
	\frac{\partial f_1}{\partial x_1}(p) & \dots & 	\frac{\partial f_1}{\partial x_n}(p) \\
	\vdots & & \vdots \\
	\frac{\partial f_{n'}}{\partial x_1}(p) & \dots & 	\frac{\partial f_{n'}}{\partial x_n}(p)
	\end{matrix} \right) \; 
	\left(\begin{matrix}
	\dot{\gamma}(0)_1 \\
	\vdots \\
	\dot{\gamma}(0)_n
	\end{matrix}\right).
\end{equation}
This allows us to define the \textit{differential} or \textit{pushforward} of $f$, as the map $d_pf$ between the tangent spaces given by the Jacobian matrix of $f$ evaluated at a point $p$, which we denote by $J_pf$, i.e.,
\[
\begin{aligned}
d_p f \;:\; T_p\manif \; &\longrightarrow \; T_{f(p)}\manif' \\
v \; &\longmapsto \; J_pf\, v.
\end{aligned}
\]
Consider two smooth maps $f:\manif \rightarrow \manif'$ and $g:\manif' \rightarrow \manif''$. The differential of the composition $(g\circ f)$ can be computed by the chain rule,
\begin{equation}\label{chain_rule}
    d_p (g\circ f) = d_{f(p)}g \circ d_{p}f: T_p \manif \longrightarrow T_{(g\circ f) (p)}\manif''.
\end{equation}
For details see \cite[Chapter 1.2]{guillemin_difftop_2010} or \cite[Theorem 8.5]{tu_manifolds_2011}.

For our applications we are interested in the case $\manif'=\R$. This simplifies the picture, since the Jacobian is just the gradient. We have the following:
\begin{equation}
	\begin{aligned}
	f: \; \manif &\longrightarrow \; \R, \quad &&p  \longmapsto \; f(p)\\
	d_p f :\;  T_p\manif &\longrightarrow \; T_{f(p)}\R, \qquad &&v  \longmapsto \nabla_x f(p)^T \; v 
	%T^*_p\manif &\longleftarrow T^*_{f(p)} \R \;:\; d^*_pf, \qquad &&\ell\big(d_pf(\cdot)\big) \longmapsfrom \ell(\cdot).
	\end{aligned}
\end{equation}
	
\begin{rmk}\label{differential_functional}
	From Proposition \ref{tanspace_dim_prop}, we have that $T_q\R\cong \R$, for any $q\in\R$. Thus, the pushforward of $f$ at any point $p\in \manif$ is a linear functional on the tangent space $T_p\manif$. In other words, we have $d_p f \in T_p^*\manif$, for all $p \in \manif$.
\end{rmk}
	
\begin{defn}\label{critical_point}
We say a point $p\in \manif$ is a \textit{critical point of $f$ (on $\manif$)}, if for every curve $\gamma$ with $\gamma(0)=p$ on $\manif$, we have
\begin{equation}\label{critical_point_firstorder}
	(f\circ \gamma)' (0) = 0.
\end{equation}
Note that $f\circ\gamma: (-a,a) \rightarrow \R$, so the expression \eqref{critical_point_firstorder} can be written out as,
\begin{equation}\label{f_compose_gamma}
	(f\circ\gamma)'(0) = f(\gamma(0))' = \nabla_x f(\gamma(t))^T \dot{\gamma}(t) \; \big|_{t=0}= \sum_{i=1}^n \frac{\partial f}{\partial x_i}(p) \;\dot{\gamma}_i(0) = 0.
\end{equation}
\end{defn}
We compare the notion of criticality introduced in Subsection \ref{local_opt_cond_section} with the one above, in Subsection \ref{morse_lagrangian_section}. For now, Definition \ref{critical_point} can be thought of as a first order condition. In the next definition, we introduce what can be thought of as a second order condition.
	
\begin{defn}\label{nondeg_def}
Let  $p\in \manif$ be a critical point of $f$. Suppose
\begin{equation}\label{nondeg_expression}
	(f\circ \gamma)''(0)\neq 0,
	\end{equation}
for all curves $\gamma$ on $\manif$, such that $\gamma(0)=p$ and $\dot{\gamma}(0) \neq 0$. Then, we say $p$ is a \textit{nondegenerate} critical point. If all the critical points of $f$ are nondegenerate, we say $f$ is \textit{Morse (on $\manif$)}. Building off of \eqref{f_compose_gamma}, we expand \eqref{nondeg_expression} as
\begin{equation}\label{nondeg_expression_expanded}
\begin{aligned}
	(f\circ \gamma)''(0) &= \left(\sum_{i=1}^n \frac{\partial f}{\partial x_i}\big(\gamma(t)\big) \;\dot{\gamma}_i(t)\right)' \; \Big|_{t=0} \\
	&= \left(\sum_{j=1}^{n}\sum_{i=1}^n \frac{\partial^2 f}{\partial x_j\partial x_i}\big(\gamma(t)\big) \;\dot{\gamma}_i(t) \; \dot{\gamma}_j(t) \; +\; \sum_{i=1}^n \frac{\partial f}{\partial x_i}\big(\gamma(t)\big) \;\ddot{\gamma}_i(t)  \right)\; \Big|_{t=0} \\
	&= \dot{\gamma}(0)^T \nabla_x^2 f(p) \;\dot{\gamma}(0) \;+ \;\sum_{i=1}^n \frac{\partial f}{\partial x_i}(p) \;\ddot{\gamma}_i(0) \neq 0.
\end{aligned}
\end{equation}
\end{defn}

\subsection{Bridge to Classical Nonlinear Optimization}\label{morse_lagrangian_section}

In this subsection, we connect the notions of Subsections \ref{local_opt_cond_section} and \ref{morse_basics_section}. We consider a manifold defined as the vanishing set of real polynomials.
	
\begin{prop}\label{regular_level_set_thm}
Given polynomials $h_1,\dots,h_m \in \R[x]$, consider the set 
\[
\sphs = \{x\in \R^n \; | \; h_i(x)=0\; \text{for each} \; i \in [m]\}\subset \R^n.
\]
If $\sphs$ is nonempty and the set of gradient vectors $\{\nabla_x h_i(p)\}_{i=1}^m$ is linearly independent for all $p\in \sphs$, then $\sphs$ is a manifold of dimension $n-m$.
\end{prop}

The proposition above is an easy consequence of the Regular Level Set Theorem, \cite[Theorem 9.9]{tu_manifolds_2011}, which states that the level set of a smooth map is a manifold, if its Jacobian has full rank on the whole level set. The assumption that the set of gradient vectors $\{\nabla_x h_i(p)\}_{i=1}^m$ is linearly independent ensures that the Jacobian has full rank on the level set $\sphs = h^{-1}(0) = (h_1^{-1}(0),\dots,h_m^{-1}(0))$.

Recall the Lagrangian from Subsection \ref{local_opt_cond_section}
\[
\mathcal{L}(x,\lambda) = f(x) - \sum_{i=1}^{m} \lambda_i h_i(x),
\]
which is a smooth function from $\R^{2n}$ to $\R$. The following lemma gives an equivalent formulation of the tangent space, in the context where the manifold is given by the vanishing locus of polynomials.
	
\begin{lem}\label{tanspace_lemma}
	With $\sphs$ as in Proposition \ref{regular_level_set_thm}, we have the following formulation of the tangent space at a point $p\in \sphs$ 
	\[T_p\sphs = \big\{v\in \R^n \; | \; \nabla_x h_i(p)^T v = 0,\; \; \text{for each}\; i \in [m]\big\}.\]
\end{lem}
\begin{proof}
	We refer to the set on the right hand side as $\mathcal{T}$. Fix $v \in T_p\sphs$ and let $\gamma$ be a representing curve on $\mathcal{S}$ for $v$, i.e., $\gamma(0)=p$ and $\dot{\gamma}(0)=v$. We have $h_i(\gamma(t))=0$ for all $t \in (-a,a)$ and $i\in [m]$. Thus, for any $i\in [m]$, we have $(h_i \circ \gamma)'(t) = 0$. We expand this expression for $t=0$,
	\[0 = (h_i \circ \gamma)'(0) = h_i(\gamma(0))' = \sum_{i=1}^n \frac{\partial h}{\partial x_i}(p) \;\dot{\gamma}_i(0)=\nabla_x h_i(p)^T \dot{\gamma}(0).\]
	So we have that $v\in \mathcal{T}$. This proves the inclusion $T_p\sphs \subseteq \mathcal{T}$. Conversely, note that $\mathcal{T}=\bigcap_{i=1}^m \big(\nabla_x h_i(p)\big)^\perp$. Since the set of gradient vectors $\{\nabla_x h_i(p)\}_{i=1}^m$ is linearly independent, we have $\dim(\mathcal{T}) = n-m$, which, by Proposition \ref{tanspace_dim_prop}, finishes the proof.
\end{proof}
	
Note, that there exists proofs in the literature where the second containment is shown by constructing a curve to represent the tangent vector $v$, see \cite[Proposition 3.23]{lee_smoothmani_2013} or \cite[Proposition 8.16]{tu_manifolds_2011}.

The next two results relate the concept of criticality in the context above to that of Subsection \ref{local_opt_cond_section}. We start with the first order optimality condition \eqref{FOOC}. 

\begin{prop}\label{criticalpoint_lagrangian}
	With $\sphs$ as in Proposition \ref{regular_level_set_thm}, consider a smooth function $f$ from $\R^n$ to $\R$. A point $p\in \mathcal{S}$ is a critical point of $f$ if and only if there exist Lagrange multipliers $\lambda=(\lambda_1,\dots,\lambda_m)$ such that $\nabla_x \mathcal{L}(p,\lambda) =0$.
\end{prop}
\begin{proof}
	Fix $v\in T_p\mathcal{S}$ and let $\gamma$ be a representing curve on $\mathcal{S}$. If $p$ is critical, then, by definition, we have $(f\circ\gamma)'(0)=0$. By the computation in \eqref{f_compose_gamma}, this implies $\nabla_x f(p)^T v =0$. Since $v$ was arbitrary, we have that $\nabla_x f(p) \in T_p\mathcal{S}^\perp$. By Proposition \ref{tanspace_lemma}, this means that $\nabla_x f(p)$ is in the span of $\{ \nabla_x h_i(p) \; | \; i \in [m]\}$. Thus, there exists scalars $\lambda_1,\dots, \lambda_m$ such that $\nabla_x f(p) = \sum_{i=1}^{m} \lambda_i \nabla_x h_i (p)$, which gives $\mathcal{L}(p,\lambda) = 0$. All the steps are reversible, and the claim follows.
	\end{proof}

The next result relates the notion of nondegeneracy to the second order sufficiency condition \eqref{SOSC}.
	
\begin{prop}\label{criticalpoint_nondeg_lagrangian}
	With $\sphs$ as in Proposition \ref{regular_level_set_thm}, let $p\in \mathcal{S}$ be a critical point of $f$ with Lagrange multipliers $\lambda= (\lambda_1,\dots, \lambda_m)$. The point $p$ is nondegenerate if and only if $v^T \nabla_x^2 \mathcal{L}(p,\lambda)v \neq 0$, for all nonzero tangent vectors $v\in T_p\mathcal{S} \setminus \{0\}$.
\end{prop}
\begin{proof}
Fix $v\in T_p\mathcal{S}\setminus \{0\}$ and let $\gamma$ be a representing curve on $\mathcal{S}$. Note that, $(f \circ \gamma)(t) = \mathcal{L}(\gamma(t), \lambda)$, since $h_i(q) = 0$ for all $q \in \mathcal{S}$ and $i\in [m]$. Thus, 
	\[
	(f\circ \gamma)''(0) = \left(\frac{\partial^2}{\partial t^2}\,\mathcal{L}(\gamma(t), \lambda) \right)\; \big|_{t=0}.
	\]
By a computation similar to that in \eqref{nondeg_expression_expanded}, we have 
	\[(f\circ \gamma)''(0) = v^T \nabla_x \mathcal{L}(p,\lambda) \; v  \; + \; \sum_{j=1}^{n} \left( \frac{\partial f}{\partial x_j} \big(p\big) \; - \; \sum_{i=1}^{m} \lambda_i\frac{\partial h_i}{\partial x_j}\big(p\big) \right) \ddot{\gamma}_j(0).\]
The second summand on the right hand side is zero since $p$ is a critical point. By the definition of nondegeneracy (Definition \ref{nondeg_def}) the claim follows.
\end{proof}

\subsection{Parametric Families}\label{morse_family_section}

In this subsection, we introduce results regarding genericity in differential geometry. We begin with Sard's classical Theorem about the critical points of a smooth map.

\begin{thm}[Sard's Theorem]\label{sard}
Let $f: U \rightarrow \R^n$ be a smooth map defined on an open set $U\subset \R^l$. Consider the set of critical points of $f$,
\[ 
\mathcal{C}_U =\{p \in U \;|\; \rank J_p f < n\}.
\]
Then the image $f(\mathcal{C}_U)\subset \R^n$ has Lebesgue measure zero.
\end{thm}

This theorem comes from \cite{sard_crit_1942} following an earlier work of Morse \cite{morse_crit_1939}\footnote{This paper is by Anthony Morse, who is \textit{not} the mathematician that Morse Theory is named after. That would be Marston Morse.}. A proof can be found in \cite{milnor_topdiff_1965}, or see \cite{federer_geomsr_1969} for a proof of a more general statement. Intuitively, Sard's theorem states that even though a smooth function $f$ may have many critical points in the domain, the set of the values of those points is relatively small in the codomain.

We now consider a parametric family of smooth functions on a manifold $\manif\subset\R^n$, represented by the smooth map $F:\R^c \times \R^n \rightarrow \R$, where \[F(b,\cdot) = f_b:\R^n\longrightarrow\R, \quad b \in \R^c\] is smooth. We assume $F$ is linear in the first arguement, i.e., $F(ab_1+b_2,\cdot) = af_{b_1} + f_{b_2}$, with $a\in\R$.

\begin{ex}\label{coefficient_space_ex}
The parameter space we are interested in is the space of coefficients of the set of multihomogeneous polynomials $\R[\bm{x}]_{=(d_1,\dots,d_m)}$. Recall that the dimension of the vector space of homogeneous polynomials in $n$ variables of degree $d$ is $\binom{n+d-1}{n-1}$, hence we have that
\begin{equation}\label{dim_ofcoefficients}
c = \prod_{i=1}^m \binom{n_i + d_i - 1}{n_i -1}.
\end{equation}
Thus, for $b\in \R^c$, we associate a multihomogeneous polynomial 
\[
F(b,\bm{x}) = f_b(\bm{x})\;=\sum_{\substack{\alpha_1 \in \N^{n_1} \dots \,\alpha_m \in \N^{n_m} \\ |\alpha_1| = d_1 \dots \,|\alpha_m| = d_m}} \; b_\alpha \bm{x}_1^{\alpha_1} \,\dots\,\bm{x}_m^{\alpha_m}
\]
which is of multidegree $(d_1,\dots,d_m)$. In this case, $F$ is linear in the first argument.
\end{ex}

Recall from Remark \ref{differential_functional}, that the differential of a real valued function on a manifold $\manif$ gives a linear functional of the tangent space. For each fixed $p\in \manif$, we define the map $\Phi_p$ from parameter space $\R^c$ to the cotangent space $T^*_p\manif$, where
\begin{equation}\label{Phi_p}
\begin{aligned}
    \Phi_p : \; \R^c &\longrightarrow T_p^*\manif \\
      b &\longmapsto d_p f_b \qquad 
    \Phi_p(b)(v) = d_pf_b(v) = \nabla_x f_b (p)^T \; v, \; \text{with} \; v\in T_p\manif.
\end{aligned}
\end{equation}
From the assumption of linearity of $F$ in the first argument, we have that $\Phi_p$ is also linear. Theorem \ref{sard} has spawned a variety of alternatives coming from different fields, see for example, \cite[Theorem 9.6.2]{bochnak_real_1998} for a statement in the context of semialgebraic geometry. It has also been used in the closely related field of semialgebraic programming, see the discussion in the introduction.  The following version can be found in \cite[Theorem 1.21]{nicolaescu_morse_2011} or \cite[Chapter 1.7]{guillemin_difftop_2010}, where proofs are given.
\begin{thm}[{\cite[Corollary 1.26]{nicolaescu_morse_2011}}]\label{nondeg_generic}
	Suppose that for all points $p$ on a manifold $\manif$ the map $\Phi_p$, defined in \eqref{Phi_p}, is surjective, and that $c \geq \dim \manif$. Then there exists a Lebesgue measure zero set $\mathcal{Z}\subset \R^c$, such that for all $b\notin \mathcal{Z}$ the function $f_b$ is Morse on $\manif$.
\end{thm}
The above theorem will be the central tool in the proof of our main result in Section \ref{finiteconvergence_sec}. 

\begin{rmk}\label{measure_zero_rmk}
If the Lebesgue measure zero set $\mathcal{Z}$ in Theorem \ref{nondeg_generic} is described by first-order formulas, then it is semialgebraic. Note that this is the case when the family $f_b$ consists of polynomials and $\mathcal{Z}$ describes the polynomials in the family which fail to be Morse. We have by \cite[Proposition 2.8.2]{bochnak_real_1998}, that taking the closure (Euclidean or Zariski) of a semialgebraic set does not increase its dimension. Thus, if the semialgebraic set $\mathcal{Z}$ is Lebesgue measure zero, it is not full dimensional. So in our case, $\mathcal{Z}$ is contained in, at most, an algebraic hypersurface.
\end{rmk}

\section{Finite Convergence on the Product of Spheres}\label{finiteconvergence_sec}
	
Our goal is to show that the moment-SOS hierarchy has finite convergence on the polynomial optimization problem \eqref{POPmultisphere}. To accomplish this, we would like to apply the finite convergence result, Theorem \ref{KKTfiniteconvergence}. We denote the feasible set of \eqref{POPmultisphere} by
\[
\sphs = \big\{\bm{x} = (\bm{x}_1,\,\dots\, ,\bm{x}_m)\; | \; s_i(\bm{x}) = 1 - \|\bm{x}_i\|^2 =0,\; i \in [m] \big\}.
\]
With $\sph^{n_i-1}\subset \R^{n_i}$ as the unit sphere, and $N=\sum_{i=1}^m n_i$, we have
\[
\sphs =\sph^{n_1-1} \times \dots \times \sph^{n_m -1} \subset \R^{n_1}\times \dots \times \R^{n_m} \cong \R^N.
\]In this section, we show that all the assumptions of Theorem \ref{KKTfiniteconvergence} are generically satisfied. To start, $(h)+\QM(g) = (s) + \Sigma[\bm{x}]$  is archimedean, as it contains the polynomial 
\[ m - \|\bm{x}\|^2 = \sum_{i=1}^m s_i(\bm{x}) = \sum_{i=1}^m (1 - \|\bm{x}_i\|^2).\]
Furthermore, we have to show that the local optimality conditions \eqref{LIC}, \eqref{SCC}, and \eqref{SOSC} hold at each global minimizer of \eqref{POPmultisphere}. Since there are no inequality constraints, \eqref{SCC} holds trivially. Moreover, the gradients of the constraints are given by the formula
\begin{equation}\label{constraintsLI}
	\nabla_x s_1 = 
	\, -2\left(\begin{matrix}
	x_{11} \\
	\vdots \\
	x_{1n_1} \\
	0 \\
	\vdots \\
	0
    \end{matrix}\right)
    ,\dots,
	\nabla_x s_i = 
	\, -2\left(\begin{matrix}
	0 \\
	\vdots \\
	0 \\
	x_{i1} \\
	\vdots \\
	x_{in_i} \\
	0 \\ 
	\vdots \\
	0
	\end{matrix}\right)
	,\dots,
	\nabla_x s_m = 
	\, -2\left(\begin{matrix}
	0 \\
	\vdots \\
	0 \\
	x_{m1} \\
	\vdots \\
	x_{mn_m} 
	\end{matrix}\right).
\end{equation}
This immediately shows that \eqref{LIC} also holds at every point of $\sphs$, because a sphere does not contain the zero vector. Note that, by Proposition \ref{regular_level_set_thm}, this implies that $\sphs\subset\R^N$ is a manifold.

What remains to show is that \eqref{SOSC} holds at every global minimizer. Recall from Proposition \ref{local_min_conditions}, that \eqref{SONC} holds at every local minimizer of an optimization problem, where \eqref{LIC} is satisfied. Hence, we only need to show that the inequality in \eqref{SONC} is sharp, i.e., for any global minimizer $\bm{p}$ of \eqref{POPmultisphere}, we need to show
\begin{equation}\label{NTS}
v^T \nabla_x^2 \mathcal{L}(\bm{p},\lambda)v \neq 0, \quad\text{for all}\quad v\in \Bigl( \bigcap_{i\in [m]}\nabla_x s_i(\bm{p})^\perp \Bigr) \setminus \{0\},   
\end{equation}
where $\lambda = (\lambda_1,\dots,\lambda_m)$ is the vector of Lagrange multipliers of $\bm{p}$. We show \eqref{NTS} by using the Sard type result Theorem \ref{nondeg_generic}. The consequence of Theorem \ref{nondeg_generic} can be understood in the context of local optimality conditions through Section \ref{morse_lagrangian_section}, specifically Proposition \ref{criticalpoint_nondeg_lagrangian}. In short, Theorem \ref{nondeg_generic} gives that under some assumptions from differential geometry a generic objective function $f$ is Morse, and $f$ being Morse means that all the local minimizers satisfy \eqref{NTS}. Note that the result is stronger than needed, as it shows that condition \eqref{NTS} holds for all local minimizers (and even all critical points), a superset of the global minimizers. To show that \eqref{NTS} holds generically, we will need the following proposition and lemma.

\begin{prop}\label{assumptions}
When considering the problem \eqref{POPmultisphere}, the assumption of surjectivity in Theorem \ref{nondeg_generic} holds. More precisely, let $\bm{p}\in \sphs = \sph^{n_1-1} \times \dots \times \sph^{n_m -1}$, and let $\R^c$ be the space of coefficients of the multihomogeneous polynomial ring $\R[\bm{x}]_{=(d_1,\dots,d_m)}$. Consider the map $\Phi_{\bm{p}}$ from $\R^c$ to the cotangent space $T^*_{\bm{p}}\sphs$, sending a vector of coefficients $b$ to the differential of the multihomogeneous polynomial $f_b$ evaluated at $\bm{p}$, as in \eqref{Phi_p}. We have that $\Phi_{\bm{p}}$ is surjective for all $\bm{p}\in \sphs$.
\end{prop}

Recall that, we assume $n_i \geq 2$ for each $i$. The following lemma shows that the dimension of the coefficient space $c$, given in \eqref{dim_ofcoefficients}, is larger than the dimension of the manifold $\sphs$, which is equal to $n-m$, for any choice of the $d_i$'s.

\begin{lem}\label{cgeqdim_lemma}
Consider the integers $n_1,\dots,n_m,d_1,\dots,d_m,m \in \N$, such that $m\geq 1$, $d_i \geq 1$, and $n_i \geq 2$ for all $i$ in $[m]$. Then we have
\[
\sum_{i=1}^m (n_i - 1)  < \prod_{i=1}^m \binom{n_i + d_i - 1}{n_i -1}.
\]
\end{lem}
\begin{proof}
Since the right hand side is increasing in $d_i$, it suffices to prove the claim for $d_1,\dots,d_m=1$. In this case, the right hand side is equal to $\prod_{i=1}^m n_i$. We proceed by induction on $m$, with the base case being trivial. The inductive hypothesis gives,
\begin{align*}
\sum_{i=1}^m (n_i - 1) &= \sum_{i=1}^{m-1} (n_i - 1) + (n_m - 1) < \prod_{i=1}^{m-1} n_i + (n_m -1) \\
& < \prod_{i=1}^{m-1} n_i + (n_m -1)\prod_{i=1}^{m-1} n_i = \prod_{i=1}^m n_i.
\end{align*}
\end{proof}

Thus, the assumption that $c \geq \dim \sphs$ in Theorem \ref{nondeg_generic} holds for \eqref{POPmultisphere}. We now need to show the surjectivity of the map $\Phi_{\bm{p}}$. We start by showing surjectivity at one point. Set
\begin{equation}\label{pn}
    \pn \; = \; \left(
    \left(\begin{matrix}
	1 \\
	\bm{0}
	\end{matrix}\right),
	\dots ,
	\left(\begin{matrix}
	1 \\
	\bm{0}
	\end{matrix}\right)
	\right),
\end{equation}
where $\bm{0}$ is the all zeros vector of the appropriate size. We fix the notation of Proposition \ref{assumptions} for the remainder of the section.
\begin{lem}\label{Phi_pn_surjective_lemma}
The map $\Phi_{\pn}$ is surjective.
\end{lem}
\begin{proof}
Let $N=\sum_{i=1}^m n_i$. We first describe the tangent space $T_{\pn}\sphs$. From Lemma \ref{tanspace_lemma}, we have
\begin{equation}\label{tanspace_description}
    T_{\pn}\mathcal{S} = \{v\in \R^N \; | \; \nabla_x s_i(\pn)^T v = 0 \text{ for all } i \in [m] \}.
\end{equation}
With $\pn$ as in \eqref{pn}, we have that 
\[
\nabla_x s_i(\pn) = \left(\bm{0},\dots ,
	\left(\begin{matrix}
	1 \\
	\bm{0} 
	\end{matrix}\right),
	\dots,\bm{0}\right) 
	\in \R^{n_1}\times \dots \times \R^{n_m} \cong \R^N,
\]
where the nonzero entry is in the $i^{th}$ horizontal position. This allows us to describe a basis of $T_{\pn}\mathcal{S}$. With $i\in [m]$ and $j \in [n_i]\setminus \{1\}$, set 
\begin{equation}\label{e_ij}
    e_{ij} \; = \; \left(
    \left(\begin{matrix}
	0 \\
	\vdots \\
	0
	\end{matrix}\right),
	\dots ,
	\left(\begin{matrix}
	\bm{0}\\
	1 \\
	\bm{0}
	\end{matrix}\right)
	, \dots,
	\left(\begin{matrix}
	0\\
	\vdots \\
	0
	\end{matrix}\right)
	\right) \; \in \; \R^{n_1}\times \dots \times \R^{n_m},
\end{equation}
where the one is in the $i^{th}$ entry `horizontally' and the $j^{th}$ entry `vertically', and $\bm{0}$ is the all zeros vector of the appropriate size. Excluding $j=1$ accounts for the orthogonality condition in \eqref{tanspace_description}. Thus, the set 
\[
\bar{\mathcal{E}} = \left\{ e_{ij} \; |\; i \in [m] , \;  j \in [n_i]\setminus\{1\} \right\}
\]
is a basis for $T_{\pn}\sphs$. The codomain of the map $\Phi_{\pn}$ is the cotangent space $T_{\pn}^*\sphs$, which is defined as the algebraic dual of $T_{\pn}\sphs$. Since the tangent space is a finite dimensional real vector space, it is isomorphic to its dual. Through an abuse of notation, we also refer to the set $\bar{\mathcal{E}}$ as a basis of $T_{\pn}^*\sphs$.

Now, to prove the claim, it suffices to show that for all $e_{ij}\in \bar{\mathcal{E}}$, there exists $b_{ij} \in \R^c$ such that $\Phi_{\pn}(b_{ij})= e_{ij}$. Fix $i\in [m]$ and $j\in [n_i]\setminus\{1\}$ and consider the monomial 
\[
x_{1,1}^{d_1} \;\dots\; x_{i-1,1}^{d_{i-1}} \;\; x_{i,1}^{d_i-1} \;\; x_{i,j} \;\; x_{i+1,1}^{d_{i+1}} \;\dots\; x_{m,1}^{d_m},
\]
and let $b_{ij}$ be the corresponding coefficient vector. Note that $b_{ij}$ has only one non zero entry, which is equal to one. To compute
\[
\Phi_{\pn}\,(b_{ij})\; = \; d_{\pn} \left(x_{1,1}^{d_1} \;\dots\; x_{i-1 ,1}^{d_{i-1}} \;\; x_{i,1}^{d_i -1} \;\; x_{i,j} \;\; x_{i+1, 1}^{d_{i+1}} \;\dots\; x_{m,1}^{d_m} \right),
\]
we note that the right hand side above is a linear functional, which can be represented as a vector in $\R^{n_1}\times \dots \times \R^{n_m}$. The entries of this vector are given by, 
\[
\left(\frac{\partial}{\partial x_{kl}} \; x_{1,1}^{d_1} \;\dots\; x_{i-1,1}^{d_{i-1}} \;\; x_{i,1}^{d_i -1} \;\; x_{i,j} \;\; x_{i+1,1}^{d_{i+1}} \;\dots\; x_{m,1}^{d_m}\;\right) \Big|_{\pn} 
\]
with $k\in [m]$ and $l \in [n_k]$. Thus, by the definition of the point $\pn$, the above expression evaluates to
\[
\Phi_{\pn}\,(b_{ij}) = \; \begin{cases}
1 \quad \text{if} \; k=i \; \text{and} \; l=j, \\
0 \quad \text{else}.
\end{cases}\\
\]
So $\Phi_{\pn}$ maps $b_{ij}$ to $e_{ij}$, and surjectivity follows from linearity of $\Phi_{\pn}$.
\end{proof}

We show surjectivity at an arbitrary point $\bm{p}$ in $\sphs$ by translating the result of the above lemma to other points of $\sphs$. Translating requires us first, to move points on the manifold, and then to move their tangent vectors. The next lemma deals with the former.

\begin{lem}\label{rotation_lemma}
Consider two points $\pn$ and $\bm{p}$ in $\sphs$, with $\pn$ as in \eqref{pn} and $\bm{p}$ arbitrary. There exists an isometric isomorphism $\varphi_{\bm{p}}:\sphs\rightarrow\sphs$ such that,
\begin{enumerate}
    \item $\varphi_{\bm{p}}(\pn)=\bm{p}$, and
    \item the differential of $\varphi_{\bm{p}}$ has full rank at every point of $\sphs$.
\end{enumerate}
\end{lem}
\begin{proof}
We give an example of such a isomorphism. Write 
\[
\pn = (\pn_1,\dots \pn_m) \quad \text{and} \quad \bm{p} = ( \bm{p}_1,\dots,\bm{p}_m),
\]
where $\pn_i$ and $\bm{p}_i$ are in $\sph^{n_i-1}$. Fix $i \in [m]$, we define a linear map $\varphi_{\bm{p}_i}(\bm{x}_i) = M_{\varphi_{\bm{p}_i}}\bm{x}_i$ which sends $\pn_i$ to $\bm{p}_i$. We set $\varphi_{\bm{p}_i}$ as the Householder transformation \cite{householder_reflection_1958} from $\R^{n_i}$ to itself, which is given by the matrix
\begin{equation}\label{householder}
M_{\varphi_{\bm{p}_i}} = I_{n_i} - 2uu^T, \;\;\text{where}\;\; u = \frac{\pn_i -\bm{p}_i}{\|\pn_i -\bm{p}_i\|}.
\end{equation}
If $\pn_i = \bm{p}_i$, we set $u =0$. It is clear that $\varphi_{\bm{p}_i}$ is an isometric isomorphism. Let $\varphi_{\bm{p}} = (\varphi_{\bm{p}_1},\dots,\varphi_{\bm{p}_m})$ be the linear function given by
\[
\varphi_{\bm{p}}(\bm{x}) =  M_{\varphi_{\bm{p}}} \bm{x},
\]
where $M_{\varphi_{\bm{p}}}$ is an $N\times N$ block diagonal matrix with blocks $M_{\varphi_{\bm{p}_i}}$ along the diagonal, making $\varphi_{\bm{p}}$ an isometric isomorphism sending $\pn$ to $\bm{p}$. Furthermore, since $\varphi_{\bm{p}}$ is a linear map, its Jacobian is also equal to $M_{\varphi_{\bm{p}}}$. Thus, the differential of $\varphi_{\bm{p}}$ has full rank.
\end{proof}

To translate the tangent vectors of points we need the next technical result, which will allow us to use the chain rule.

\begin{lem}\label{phi_keeps_mhomo}
Let $\varphi_{\bm{p}}$ be as in the proof of Lemma \ref{rotation_lemma}. Fix a multihomogeneous polynomial $f_{\bar{b}}$ in $\R[\bm{x}]_{=(d_1,\dots,d_m)}$, with coefficient vector $\bar{b}\in \R^c$. We have that the composition $f_{\bar{b}} \circ \varphi_{\bm{p}}$ is a multihomogeneous polynomial of the same multidegree as $f_{\bar{b}}$, i.e., 
\[\text{there exists}\;\; b\in \R^c \; \st \; f_b = f_{\bar{b}} \circ \varphi_{\bm{p}}\in \R[\bm{x}]_{=(d_1,\dots,d_m)}.\]
Furthermore, the analogous statement also holds for the inverse $\varphi_{\bm{p}}^{-1}$.
\end{lem}
\begin{proof}
We recall some notation. We denote variables as $\bm{x} = (\bm{x}_1 , \dots , \bm{x}_m)$ with $\bm{x}_i = (x_{i1},\dots ,x_{in_i})$, for each $i\in [m]$. For $f \in \R[\bm{x}]_{=(d_1,\dots,d_m)}$, we write $f(\bm{x}) \,= \,\sum \,f_\alpha \,\bm{x}^{\alpha_1}_1 \,\dots\, \bm{x}^{\alpha_m}_m$, where $\alpha_i = (\alpha_{i1},\dots, \alpha_{in_i})$ and $|\alpha_i|=d_i$, for each $i$ in $[m]$, and so $\bm{x}_i^{\alpha_i} = x_{i1}^{\alpha_{i1}} \dots x_{in_i}^{\alpha_{in_i}}$. From Lemma \ref{rotation_lemma}, we have $\varphi_{\bm{p}} = (\varphi_{\bm{p}_1}, \dots, \varphi_{\bm{p}_m})$, where for each $i$ in $[m]$, $\varphi_{\bm{p}_i}$ is a linear map from $\R^{n_i}$ to itself. In other words, $\varphi_{\bm{p}_i}(\bm{x}_i)$ is a length $n_i$ vector where each entry is a linear polynomial in the variables $x_{i 1},\dots,x_{i n_i}$. Writing out $(f\circ\varphi_{\bm{p}}) (\bm{x})$, we have
\[
\begin{aligned}
f(\varphi_{\bm{p}}(\bm{x})) \, &=\, f\big(\varphi_{\bm{p}_1}(\bm{x}_1),\dots, \varphi_{\bm{p}_m}(\bm{x}_m)\big)\\
&= \, \sum \, f_\alpha \,\big(\varphi_{\bm{p}_1}(\bm{x}_1)\big)^{\alpha_1} \dots \big(\varphi_{\bm{p}_m}(\bm{x}_m)\big)^{\alpha_m}.
\end{aligned}
\]
Fix $i$ in $[m]$. Since $\varphi_{\bm{p}_i}(\bm{x}_i)$ is a vector with entries being linear polynomials, we have
\[
\big(\varphi_{\bm{p}_i}(\bm{x}_i)\big)^{\alpha_i} = \big(\ell_{i1}(x_{i 1},\dots,x_{i n_i})\big)^{\alpha_{i1}} \dots \, \big(\ell_{in_i}(x_{i 1},\dots,x_{i  n_i})\big)^{\alpha_{in_i}},
\]
where $\ell_{ij}$ is a homogeneous polynomial of degree one, for each $j$ in $[n_i]$. Every monomial in the product above will have degree $d_i=|\alpha_i|$, and hence the polynomial $f\circ \varphi_{\bm{p}}$ will be multihomogeneous of multidegree $(d_1,\dots,d_m)$.

Since the inverse of a block diagonal matrix is also block diagonal, the argument above is the same for $\varphi_{\bm{p}}^{-1}$.
\end{proof}

\begin{lem}\label{surjectivity_lemma}
Consider two points $\pn$ and $\bm{p}$ in $\sphs$, with $\pn$ as in \eqref{pn} and $\bm{p}$ arbitrary. If $\Phi_{\pn}$ is surjective, then so is $\Phi_{\bm{p}}$.
\end{lem}
\begin{proof}
Let $\varphi_{\bm{p}}^{-1}$ be the inverse of the isomorphism in the proof of Lemma \ref{rotation_lemma}, which sends $\bm{p}$ to $\pn$ and has the representing matrix $M_{\varphi_{\bm{p}}}^{-1}$, as in \eqref{householder}.  As a linear map, the Jacobian of $\varphi_{\bm{p}}^{-1}$ is equal to its representing matrix. Thus, the differential is surjective and given by,\[\begin{aligned}
d_{\bm{p}}\varphi_{\bm{p}}^{-1}: T_{\bm{p}}\sphs \; &\longrightarrow \; T_{\pn}\sphs\\
v\; &\longmapsto \; M^{-1}_{\varphi_{\bm{p}}} \,v.
\end{aligned}\]

For a linear functional $\ell \in T_{\pn}^*\sphs$, we can send it to $T^*_{\bm{p}}\sphs$ by precomposing with the differential map above, i.e., $\ell(\cdot) \mapsto \ell \big( M^{-1}_{\varphi_{\bm{p}}} \,(\cdot)\big)$. This map is referred to as the pullback of $\varphi_{\bm{p}}^{-1}$ in differential geometry.

We show the map $\Phi_{\bm{p}}$ is surjective, by showing that for each $\ell\in T^*_{\bm{p}}\sphs$, there exists a coefficient vector $b\in\R^c$, such that $\Phi_{\bm{p}}(b)=d_{\bm{p}}f_b=\ell$.

Fix $\ell \in T^*_{\bm{p}}\sphs$ such that $\ell(\cdot)= \bar{\ell}\,\big(M^{-1}_{\varphi_{\bm{p}}} \,(\cdot)\big)$, for some $\bar{\ell}\in T^*_{\pn}\sphs$. By the surjectivity of $\Phi_{\pn}$, we have, for some $\bar{b}\in \R^c$,
\[
\Phi_{\pn}(\bar{b}) \;=\; d_{\pn}f_{\bar{b}}\;=\;\bar{\ell}.
\]
Since $\varphi_{\bm{p}}^{-1}(\bm{p})=\pn$, we have by the chain rule \eqref{chain_rule} that \[\ell=d_{\pn}f_{\bar{b}}\big( M^{-1}_{\varphi_{\bm{p}}} \,(\cdot)\big) = d_{\pn}f_{\bar{b}}\big( d_{\bm{p}}\varphi_{\bm{p}}^{-1} \,(\cdot)\big) =d_{\bm{p}}\big( f_{\bar{b}} \circ \varphi^{-1}_{\bm{p}}\big).\] By Lemma \ref{phi_keeps_mhomo}, we have that there exists $b\in\R^c$ such that $\ell=d_{\bm{p}}f_b = \Phi_{\bm{p}}(b)$, as required.
\end{proof}

We are now ready to prove Proposition \ref{assumptions}.

\begin{proof}[Proof of Proposition \ref{assumptions}]
We want to show that for arbitrary $\bm{p}\in \sphs$, the map $\Phi_{\bm{p}}$ is surjective. Lemma \ref{Phi_pn_surjective_lemma} gives that $\Phi_{\pn}$ is surjective. Thus, the assumption of Lemma \ref{surjectivity_lemma} is satisfied and we have that $\Phi_{\bm{p}}$ is surjective for all $\bm{p}$ in $\sphs$.
\end{proof}

Proposition \ref{assumptions} and Lemma \ref{cgeqdim_lemma} show that the assumptions of Theorem \ref{nondeg_generic} hold for the polynomial optimization problem \eqref{POPmultisphere}. This allows the application of Theorem \ref{nondeg_generic} to prove our main result.

\begin{proof}[Proof of Theorem \ref{main_result}]
When considering the optimization problem \eqref{POPmultisphere}, the assumptions of Theorem \ref{nondeg_generic} are satisfied by Proposition \ref{assumptions} and Lemma \ref{cgeqdim_lemma}. Hence, by the result of Theorem \ref{nondeg_generic}, a generic multihomogeneous polynomial $f$ is Morse. By Proposition \ref{criticalpoint_nondeg_lagrangian}, or the discussion preceding Proposition \ref{assumptions}, $f$ being Morse means that \eqref{SOSC} holds at every global minimizer of \eqref{POPmultisphere}. Thus, the assumptions of Theorem \ref{KKTfiniteconvergence} are satisfied, and so the moment-SOS hierarchy applied to \eqref{POPmultisphere} has finite convergence. Furthermore, the optimality certificate of flat truncation for the moment hierarchy holds for a sufficiently high order of the hierarchy.
\end{proof}

\section{Inequality Constraints and Transversality}\label{inequalities_sec}

In this section, we consider the addition of inequality constraints to the polynomial optimization problem over the product of spheres \eqref{POPmultisphere}. 
Recall the notation,
\[
\sphs = \{\bm{x}\in\R^N\;|\; s_i(\bm{x}) = 0,\; i \in [m]\}= \sph^{n_1-1}\times \dots \times \sph^{n_m-1}\subset \R^N.
\]
We formulate the new problem as,
\begin{equation}\label{POPmultisphere_ineqs}
\begin{aligned}
	f_{\min} = \; &\inf \; \; f (\bm{x})\\
	& \st \;  s_i(\bm{x}) = 1 - \|\bm{x}_i\|^2 =0, \;  i \in [m],\\
	& \qquad g_j(\bm{x}) \geq 0, \; j \in [\bar{m}],
\end{aligned}
\end{equation}
where $g_1,\dots,g_{\bar{m}}\in \R[\bm{x}]$. We extend our main result by showing that the moment-SOS hierarchy has finite convergence for \eqref{POPmultisphere_ineqs}. 
We do this in more generality: we consider a linear space of polynomials with coefficient space $\R^c$, and let $f_b$ be a polynomial given by the coefficient vector $b\in\R^c$. Then we consider the more general polynomial optimization problem, 
\begin{equation}\label{POP_general_ineqs}
\begin{aligned}
	f_{\min} = \; &\inf \; \; f_b (\bm{x})\\
	& \st \;  h_i(\bm{x}) = 0, \;  i \in [m],\\
	& \qquad g_j(\bm{x}) \geq 0, \; j \in [\bar{m}],
\end{aligned}
\end{equation}
and make the following assumption on the polynomials describing the feasible set. 
\begin{assump}\label{LI_assumption}
Consider polynomials $h_1,\dots,h_m$ with vanishing set
\[
\manif = \{\bm{x}\in \R^N \;|\; h_i(\bm{x})=0 \;\;\text{for all}\;\; i\in [m]\}.
\]
Assume the set of gradient vectors $\{\nabla_x h_i(\bm{p})\}_{i=1}^m$ is linearly independent for all $\bm{p}\in \manif$. Further, for $J\subset[\bar{m}]$, consider the set
\begin{equation}\label{S_g_J}
\begin{aligned}
\manif_J = \manif \, &\cap \, \{\bm{x}\in \R^N\;|\; g_j(\bm{x})=0\;\;\text{for all}\;\; j\in J\} \\
&\cap \, \{\bm{x}\in \R^N\;|\;g_l(\bm{x}) > 0\;\; \text{for all}\;\; l\notin J\}.
\end{aligned}
\end{equation}
For any $J$, such that $\manif_J$ is nonempty, and for every $\bm{p} \in \manif_J$, we assume that the vectors $\{\nabla_x g_j(\bm{p})\}_{j \in J} \cup \{\nabla_x h_i(\bm{p})\}_{i=1}^m$ are linearly independent.
\end{assump}

The main result of this section is the following Theorem.

\begin{thm}\label{general_ineqs_thm}
Consider a linear space of polynomials with coefficient space $\R^c$, and let $f_b$ be a generic polynomial given by coefficient vector $b\in\R^c$, i.e., for $b$ outside some Lebesgue measure zero subset of $\R^c$. 
Suppose that Assumption \ref{LI_assumption} holds for \eqref{POP_general_ineqs}. 
Furthermore, suppose that $c \ge \dim \manif$ and that the map $\Phi_p$ from parameter space $\R^c$ to the cotangent space $T^*_p\manif$, where
\begin{equation}\label{Phi_p_general}
\begin{aligned}
    \Phi_p : \; \R^c &\longrightarrow T_p^*\manif \\
      b &\longmapsto d_p f_b \qquad 
    \Phi_p(b)(v) = d_pf_b(v) = \nabla_x f_b (p)^T \; v, \; \text{with} \; v\in T_p\manif,
\end{aligned}
\end{equation}
is surjective. Then, for the problem \eqref{POP_general_ineqs}, the moment-SOS hierarchy has finite convergence. Furthermore, the optimality certificate of flat truncation for the moment hierarchy holds for a sufficiently high order of the hierarchy.
\end{thm}

Theorem \ref{general_ineqs_thm} readily implies that for the problem \eqref{POPmultisphere_ineqs} the moment-SOS hierarchy has finite convergence, for a generic multihomogeneous objective function $f$. This extends our main result, Theorem \ref{main_result}.

\begin{cor}\label{main_result_ineqs_thm}
Let $f$ be a generic multihomogeneous polynomial. For the polynomial optimization problem \eqref{POPmultisphere_ineqs}, if Assumption \ref{LI_assumption} with $h_i=s_i$ holds, the moment-SOS hierarchy has finite convergence. Furthermore, the optimality certificate of flat truncation for the moment hierarchy holds for a sufficiently high order of the hierarchy.
\end{cor}

Recall that in Section \ref{finiteconvergence_sec}, we showed the map \eqref{Phi_p_general} is surjective for a generic multihomogeneous polynomial when $\manif=\sphs$. Thus, Corollary \ref{main_result_ineqs_thm} is an easy consequence of Theorem \ref{general_ineqs_thm}. The motivation for the more general statement, Theorem \ref{general_ineqs_thm}, is that it implies the following corollary about sets other than the product of spheres.

\begin{cor}\label{ineqs_cor}
Consider the polynomial optimization problem \eqref{POP_general_ineqs} and suppose that Assumption \ref{LI_assumption} holds. 
Then the moment-SOS hierarchy has finite convergence when:
\begin{enumerate}
\item The objective function is a generic polynomial. 
\item The objective function is a generic homogeneous polynomial and the manifold $\mathcal{K}$ does not contain $0$. 
\end{enumerate}
\end{cor}

The items of Corollary \ref{ineqs_cor} already follow from the techniques used in  \cite{lee_genericcpt_2016,lee_generic_2017,springarn_optcondi_1982}. Note that the second item generalizes the main result of \cite{huang_homosphere_2023}, Theorem \ref{sphere_finite}, to other sets. In particular the moment-SOS hierarchy has finite convergence for polynomial optimization over the simplex, generically.

Like in Section \ref{finiteconvergence_sec}, we use Theorem \ref{KKTfiniteconvergence} to prove finite convergence. To use Theorem \ref{KKTfiniteconvergence}, we again employ techniques from differential geometry, references for which include \cite{guillemin_difftop_2010,lee_smoothmani_2013,tu_manifolds_2011}.

\subsection{Parametric Transversality}\label{transversality_sec}

We introduce the notion of transversality, and a Sard type result discussing its genericity. We define a \textit{submanifold} $\mathcal{N}$ of $\manif$, as a subset of $\manif$ which is also a manifold of some dimension. We now define the notion of transversality, more details can be found in \cite[Chapter 1.5]{guillemin_difftop_2010} or \cite[Chapter 6]{lee_smoothmani_2013}. 
\begin{defn}\label{transversality_def}
For manifolds $\mathcal{M}$ and $\mathcal{K}$, we say a map $\psi: \mathcal{M} \rightarrow \mathcal{K}$ is \emph{transverse to a submanifold $\mathcal{N}\subset \mathcal{K}$}, if 
\[
d_p \psi\big(T_p \mathcal{M}\big) + T_{\psi(p)} \mathcal{N} = T_{\psi(p)} \mathcal{K},\quad \text{for all} \; p \; \text{in the preimage} \; \psi^{-1}(\mathcal{N}).
\]
\end{defn}
Note that in the definition above, the inclusion `$\subseteq$' always holds.

A typical example of transversality is a map which embeds a line in $\R^3$, the image of which intersects a plane in just one point. Furthermore, for a line and a plane in $\R^3$ containing the origin, suppose the intersection is such that the line is fully contained in the plane. Then a small perturbation of the entries of the vector describing the line would cause the intersection to be just the origin. The following Sard type theorem is a formalization and generalization of this example, and it is known as the Parametric Transversality Theorem, see \cite[Chapter 2.3]{guillemin_difftop_2010} or \cite[Theorem 6.35]{lee_smoothmani_2013} for details.
\begin{thm}[Parametric Transversality]\label{transversality_parametric}
Suppose $\mathcal{M}$ and $\mathcal{K}$ are manifolds, and that $\mathcal{N}\subset \mathcal{K}$ is a submanifold. Consider a smooth function 
\[F: \R^c \times \mathcal{M} \rightarrow \mathcal{K}, \;\;\text{with}\;\;f_b = F(b,\cdot) \;\; \text{for}\;\; b\in \R^c.\]
Suppose $F$ is transverse to $\mathcal{N}$. Then there exists a Lebesgue measure zero set $\mathcal{V}\subset \R^c$, such that for all $b\notin \mathcal{V}$ the function $f_b$ is transverse to $\mathcal{N}$. 
\end{thm}

Note that Remark \ref{measure_zero_rmk} also applies to the Lebesgue measure zero set $\mathcal{V}$ in the above statement.

\subsection{Finite Convergence with Inequalities}

Let $h_1,\dots,h_m,g_1,\dots, g_{\bar{m}}$ be the polynomials defining the feasible set in \eqref{POP_general_ineqs}, and suppose Assumption \ref{LI_assumption} is satisfied. Similarly to Sections \ref{morse_sec} and \ref{finiteconvergence_sec}, we let $\R^c$ be the space of coefficients of a linear space of polynomials with a basis $q_1(x), \dots, q_c(x)$. For $b\in \R^c$, we associate a polynomial via the map
\[F(b,x) = f_b(x) = \sum_{i=1}^c b_i q_i(x).\]
In the case of multihomogeneous polynomials, this map has the following form,
\[
F(b,\bm{x}) = f_b(\bm{x})\;=\sum_{\substack{\alpha_1 \in \N^{n_1} \dots \,\alpha_m \in \N^{n_m} \\ |\alpha_1| = d_1 \dots \,|\alpha_m| = d_m}} \; b_\alpha \bm{x}_1^{\alpha_1} \,\dots\,\bm{x}_m^{\alpha_m}.
\]

In this subsection, we prove Theorem \ref{general_ineqs_thm} by using the finite convergence result Theorem \ref{KKTfiniteconvergence}. We show that the assumptions of Theorem \ref{KKTfiniteconvergence}, regarding local optimality conditions (introduced in Subsection \ref{local_opt_cond_section}), hold for a generic objective function. That is, we show for all $b$ outside a Lebesgue measure zero set in $\R^c$, the conditions \eqref{LIC}, \eqref{SCC}, and \eqref{SOSC} hold at every global minimizer of \eqref{POP_general_ineqs} with objective function $f_b$.

By Assumption \ref{LI_assumption}, we have that \eqref{LIC} holds for all global minimizers. The tools developed in the previous section give \eqref{SOSC}. Note that, $\manif_J$ is a subset of the feasible set of problem \eqref{POP_general_ineqs}, and is a manifold of dimension $\dim(\manif) - |J|$, by Proposition \ref{regular_level_set_thm}.

\begin{prop}\label{SOSC_ineq_prop}
Consider the problem  \eqref{POP_general_ineqs} with objective function $f_b$, and such that assumptions of Theorem \ref{general_ineqs_thm} are satisfied. Then, there exists a Lebesgue measure zero set $\mathcal{Z}\subset \R^c$, such that for every $b\notin \mathcal{Z}$ \eqref{SOSC} holds for every global minimizer of \eqref{POP_general_ineqs}.
\end{prop}
\begin{proof}
By Lemma \ref{tanspace_lemma}, it is immediate that $T_{\bm{p}}\manif_J\subset T_{\bm{p}}\manif$, for any $J\subseteq [\bar{m}]$. Thus, we have that the map $\Phi_p$, with properly adjusted codomain, is surjective onto $T_{\bm{p}}^*\manif_J$, at every $\bm{p}\in \manif_J$ and every $J\subseteq [\bar{m}]$. 
Thus, for every $J\subseteq[\bar{m}]$, Theorem \ref{nondeg_generic} gives a Lebesgue measure zero set $\mathcal{Z}_J\subset \R^c$, such that for every $b\notin \mathcal{Z}_J$ \eqref{SOSC} holds for every local minimizer of $f_b$ on $\manif_J$. This is a superset of the set of global minimizers. The finite union of Lebesgue measure zero sets is again Lebesgue measure zero, so setting $\mathcal{Z} = \cup_{J\subset[\bar{m}]}\mathcal{Z}_J$ gives the result.
\end{proof}

The remainder of this subsection is dedicated to showing that \eqref{SCC} holds for all global minimizers, generically. For an arbitrary global minimizer $\bm{p}_0$ of \eqref{POP_general_ineqs}, we recall some of the notions introduced in Subsection \ref{local_opt_cond_section}. For $j$ in $[\bar{m}]$, let $\mu_j$ be the Lagrange multiplier of $g_j$ at $\bm{p}_0$, as in \eqref{lagrangian_function}. Further, let $J(\bm{p}_0)\subset[\bar{m}]$ be the set of the active constraints of $\bm{p}_0$. Since $\bm{p}_0$ is a local minimum, we have by Proposition \ref{local_min_conditions} that \eqref{CC} is satisfied at $\bm{p}_0$, meaning,
\begin{equation}\label{CC_specific}
\begin{aligned}
\mu_j \geq 0, \;\; &\text{for all} \;\; j \in J(\bm{p}_0),\\
\mu_j=0, \;\; &\text{for all} \;\; j \in [\bar{m}]\setminus J(\bm{p}_0).
\end{aligned}
\end{equation}
To show that \eqref{SCC} also holds at $\bm{p}_0$, we need to show that, additionally to the conditions in \eqref{CC_specific}, we generically have that 
\begin{equation}\label{SCC_specific}
\mu_j>0, \;\; \text{for all} \quad j\in J(\bm{p}_0).
\end{equation}

We give a high level description of the proof. We proceed by showing that for any $J \subseteq [\bar{m}]$, all the critical points, (a superset of the global minimizers) which have $J$ as their set of active constraints, have Lagrange multipliers satisfying \eqref{SCC_specific}, generically. To do this, we employ Parametric Transversality, Theorem \ref{transversality_parametric}. We construct a parameterized map, which captures data about critical points, and a manifold, with Lagrange multiplier data. We show that these are transverse in the sense of Definition \ref{transversality_def}. Then the result of Theorem \ref{transversality_parametric} gives transversality for a fixed generic parameter. To finish, a dimension counting argument relying on the definition of transversality shows \eqref{SCC_specific}.

\begin{defn}
For a smooth map $f: \R^n \rightarrow \R^m$, and a manifold $\mathcal{M}\subset\R^n$, if 
\begin{enumerate}
    \item the differential $d_pf$ is injective for all $p\in \mathcal{M}$, and
    \item the restriction $f_{|\mathcal{M}}$ is a homeomorphism,
\end{enumerate}
we say $f$ is an \textit{embedding} of $\mathcal{M}$.
\end{defn}
A map satisfying the first condition in the definition above is often referred to as an immersion in the literature. The following theorem gives that an embedding preserves the manifold structure.
\begin{thm}[{\cite[Theorem 11.13]{tu_manifolds_2011}}]\label{embedding_thm}
Given a manifold $\mathcal{M}\subset \R^n$ and $f:\R^n\rightarrow\R^m$ an embedding of $\mathcal{M}$, the image $f(\mathcal{M})\subset \R^m$ is also a manifold.
\end{thm}
Since a homeomorphism preserves dimension and the differential of an embedding is injective, we have that embeddings preserve tangent spaces. With the notation of Theorem \ref{embedding_thm}, we have that $d_pf(T_p\manif) = T_{f(p)}\big(f(\manif)\big)$. 
For more details, we refer to the discussion preceding \cite[Proposition 5.35]{lee_smoothmani_2013} or \cite[Lemma 1.2.10]{salamon_difftop_lecturenotes}. 

Given $J \subseteq [\bar{m}]$, let $\manif_J$ be as in Assumption \ref{LI_assumption}. Consider the map $\psi_{J} \colon \R^N \times \R^{m} \times \R^{|J|} \to \R^{2N}$ defined by
\[
\psi_{J}(\bm{p}, \lambda, \mu) = \Big(\bm{p}, \sum_{i=1}^m \lambda_i \nabla_x h_i(\bm{p}) + \sum_{j \in J} \mu_{j} \nabla_x g_{j}(\bm{p})\Big).
\]
Given a subset $\hat{J}\subseteq J$, we are interested in the image $\psi_{\hat{J}}(\manif_J \times \R^{m} \times \R^{|\hat{J}|})$, and will show it is a manifold by showing that the map $\psi_{\hat{J}}$ is an embedding of $\manif_J \times \R^{m} \times \R^{|\hat{J}|}$.

\begin{lem}\label{embedding_lemma}
For $\hat{J}\subseteq J$, the map $\psi_{\hat{J}}$ is an embedding of $\manif_J\times \R^{m} \times \R^{|\hat{J}|}$. Furthermore, $\psi_{\hat{J}}(\manif_J \times \R^{m} \times \R^{|\hat{J}|})$ is a submanifold of $\manif_J \times \R^{N}$ of dimension $N - |J| + |\hat{J}|$.
\end{lem}
\begin{proof}
Fix $\hat{J}\subseteq J$. To simplify notation, we write $|\hat{J}| =\ell$, $\hat{J} = \{j_1,\dots,j_\ell\}$, and $\psi_{\hat{J}} = \psi$. Write the map in the form
\[
\psi(\bm{p},\lambda,\mu)
=
\bigl(\bm{p}, A(\bm{p})(\lambda,\mu)\bigr),
\]
where $A(\bm{p})$ is the $N\times (m+\ell)$ matrix whose columns are the vectors
\[
\nabla_x h_1(\bm{p}),\dots,\nabla_x h_m(\bm{p}), \nabla_x g_{j_1}(\bm{p}),\dots, \nabla_x g_{j_\ell}(\bm{p}).
\]
By assumption, these vectors are linearly independent for all $\bm{p}\in \manif_J$, so $A(\bm{p})$ has full column rank.

We begin by showing the differential $d_{(\bm{p},\lambda,
\mu)}\psi$ is injective for all $(\bm{p},\lambda,
\mu)\in \manif_J \times \R^{m} \times \R^{\ell}$. Let $(v,\dot\lambda,\dot\mu)\in T_{(\bm{p},\lambda,\mu)}(\manif_J\times\R^m\times\R^\ell)$.
A calculation gives
\begin{equation}\label{diff_of_psi}
\begin{aligned}
d_{(\bm{p},\lambda,\mu)}\psi(v,\dot\lambda,\dot\mu) =
\Big(v,\;
\sum_{i=1}^m \lambda_i\nabla_x^2 h_i(\bm{p})v
\;+\;&
\sum_{j\in\hat{J}} \mu_j\nabla_x^2 g_j(\bm{p})v\\
\;+\;&
\sum_{i=1}^m \dot\lambda_i\nabla_x h_i(\bm{p})
\;+\;
\sum_{j\in\hat{J}} \dot\mu_j\nabla_x g_j(\bm{p})
\Big).
\end{aligned}
\end{equation}
If the above equals $(\bm{0},\bm{0})\in \R^N\times \R^N$, then the first component gives $v=\bm{0}$, and the second reduces to
\[
\sum_{i=1}^m \dot\lambda_i\nabla_x h_i(\bm{p})
+
\sum_{j\in\hat{J}} \dot\mu_j\nabla_x g_j(\bm{p})=0.
\]
The linear independence assumption forces $\dot\lambda=\dot\mu=0$, and hence $d_{(\bm{p},\lambda,\mu)}\psi$ is injective.

From its definition, we can see that $\psi$ is continuous. Thus, to show that the restriction of $\psi$ is a homeomorphism onto its image, we show it is injective, and has a continuous inverse. Suppose $\psi(\bm{p},\lambda,\mu)=\psi(\bm{p}',\lambda',\mu')$. From the first coordinate we obtain $\bm{p}=\bm{p}'$. From the second coordinate,
\[
\sum_{i=1}^m (\lambda_i-\lambda'_i)\nabla_x h_i(\bm{p})
+
\sum_{j\in\hat{J}} (\mu_j-\mu'_j)\nabla_x g_j(\bm{p})
=0.
\]
Linear independence of the gradients implies $\lambda=\lambda'$ and $\mu=\mu'$, so $\psi$ is injective. Since the first coordinate of $\psi$ is the identity on $\bm{p}$, any point $\psi(\bm{p},\lambda,\mu)=(\bm{p},w)$ determines $\bm{p}$ immediately. For fixed $\bm{p}$, the equation
\[
A(\bm{p})(\lambda,\mu)=w
\]
has a unique solution since $A(\bm{p})$ has full column rank. Moreover,
\[
(\lambda,\mu)
=
\Big(A(\bm{p})^T A(\bm{p})\Big)^{-1}A(\bm{p})^T w
\]
depends continuously on $(\bm{p},w)$, as it is a composition of continuous maps:

\begin{itemize}
\item $\bm{p} \mapsto A(\bm{p})$ is continuous because $\nabla h_i(\bm{p})$ and $\nabla g_j(\bm{p})$ are continuous.
\item $A(\bm{p}) \mapsto A(\bm{p})^T A(\bm{p})$ is continuous as it is obtained by continuous matrix multiplication.
\item Since $A(\bm{p})^T A(\bm{p})$ is invertible for all $\bm{p} \in \manif_J$ (full column rank), the inverse map
      $M \mapsto M^{-1}$ is continuous on invertible matrices.
\item $(\bm{p},w) \mapsto A(\bm{p})^T w$ is continuous by continuous matrix--vector multiplication.
\end{itemize}
Hence, the inverse of the restriction of $\psi$ exists and is continuous, so the restriction of $\psi$ is a homeomorphism.

Thus, the map $\psi$ is an embedding of $\manif_J\times \R^{m} \times \R^{\ell}$, and, by Theorem \ref{embedding_thm}, we have that $\psi\big(\manif_J \times \R^{m} \times \R^{\ell}\big)$ is a submanifold of $\manif_J \times \R^{N}$.

Finally, since  $\dim \manif = N-m$, we have $\dim \manif_J= (N-m)-|J|$.
Therefore,
\[
\dim\big(\manif_J\times\mathbb{R}^m\times\mathbb{R}^\ell\bigr)
=
(N-m-|J|)+m+\ell
= N-|J|+\ell.
\]
This equals the dimension of the embedded submanifold $\psi(\manif_J\times\mathbb{R}^m\times\mathbb{R}^\ell)$, by Theorem \ref{embedding_thm}.
\end{proof}

With the notation recalled at the beginning of this subsection, we consider the map 
\[
\begin{aligned}
G \colon \R^{c} \times \manif_J \; &\longrightarrow \; \manif_J \times \R^{N}\\ 
(b,\bm{p})\; &\longmapsto \;(\bm{p}, \nabla_x f_b(\bm{p})).
\end{aligned}
\]

\begin{lem}
Given $\hat{J}\subseteq J \subseteq [\bar{m}]$, the map $G$ is transverse to the submanifold $\psi_{\hat{J}}(\manif_J\times \R^{m} \times \R^{|\hat{J}|})$.
\end{lem}
\begin{proof}
As in the previous lemma, we fix $\hat{J}\subseteq J$, and simplify notation, by writing $|\hat{J}| =\ell$, $\hat{J} = \{j_1,\dots,j_\ell\}$, and $\psi_{\hat{J}} = \psi$. 

Let $(b,\bm{p})$ be such that $G(b,\bm{p}) \in \psi(\manif_J \times \R^{m} \times \R^{\ell})$. By definition of transversality, we want to show that 
\[
d_{(b,\bm{p})} G\big(T_{(b,\bm{p})} (\R^{c} \times \manif_J) \big) + T_{G(b,\bm{p})} \psi(\manif_J \times \R^{m} \times \R^{\ell}) = T_{G(b,\bm{p})} (\manif_J\times \R^{N}).
\]
The inclusion `$\subseteq$' always holds. Note the following simplifications of the first and third terms in the above expression, 
\[T_{(b,\bm{p})} (\R^{c} \times \manif_J) = \R^{c} \times T_{\bm{p}}\manif_J
\qquad
T_{G(b,\bm{p})} (\manif_J\times \R^{N}) = T_{\bm{p}} \manif_J\times \R^{N} .
\]
Moreover, since $G(b,\bm{p}) \in \psi(\manif_J \times \R^{m} \times \R^{\ell})$, there exists $(\lambda, \mu) \in \R^{m} \times \R^{\ell}$ such that $  \psi(\bm{p},\lambda,\mu) = G(b,\bm{p})$. Hence, 
\[
\begin{aligned}
T_{G(b,\bm{p})} \psi\big(\manif_J \times \R^{m} \times \R^{\ell}\big) =& \;T_{\psi(\bm{p},\lambda, \mu)} \psi\big(\manif_J \times \R^{m} \times \R^{\ell}\big)\\ 
=& \;d_{(\bm{p},\lambda, \mu)}\psi\bigl(T_{(\bm{p},\lambda, \mu)}(\manif_J \times \R^{m} \times \R^{\ell}) \bigr)\\ 
=& \;d_{(\bm{p},\lambda, \mu)}\psi\big(T_{\bm{p}}\manif_J \times \R^{m} \times \R^{\ell}\big).
\end{aligned}
\]
Combining these observations, we need to show the equality
\begin{equation}\label{trans_equality_in_pf}
d_{(b,\bm{p})} G\big(\R^{c} \times T_{\bm{p}}\manif_J\big) + d_{(\bm{p},\lambda, \mu)}\psi\big(T_{\bm{p}}\manif_J \times \R^{m} \times \R^{\ell}\big) = T_{\bm{p}} \manif_J \times \R^{N}.
\end{equation}
From \eqref{diff_of_psi} in the proof of Lemma \ref{embedding_lemma}, we have the that differential $d_{(\bm{p},\lambda, \mu)}\psi$ has the following matrix representation,
\[
\begin{pmatrix}
I_N & 0 \\
H(\bm{p},\lambda,\mu) & A(\bm{p})
\end{pmatrix},
\]
where $A(\bm{p})$ has the gradients of the constraints as its column vectors and $H=H(\bm{p},\lambda,\mu)$ is the sum of the Hessians of the constraints. The matrix representation of $d_{(b,\bm{p})} G$ has the form,
\[
\begin{pmatrix}
0 & I_N \\
M(\bm{p}) & B(b,\bm{p})
\end{pmatrix},
\]
where $M(\bm{p})$ has the property that, for every $\bar{b} \in \R^c$, we have $M(\bm{p}) \bar{b} = \nabla_x f_{\bar{b}}(\bm{p})$. 

Fix an arbitrary pair $(u,v) \in T_{\bm{p}} \manif_J\times \R^{N}$. We show that $(u,v)$ is in the sum of the images of the differentials, and hence prove \eqref{trans_equality_in_pf}. Denote $Hu = u_1 + u_2$, where $u_1$ is the projection of $Hu$ on $T_{\bm{p}}\manif$ and $u_2$ is the projection of $Hu$ on $(T_{\bm{p}}\manif)^\perp$. Likewise, write $v= v_1 +v_2$, where $v_1$ is the projection of $v$ on $T_{\bm{p}}\manif$ and $v_2$ is the projection of $v$ on $(T_{\bm{p}}\manif)^{\perp}$. Since $\Phi_p$ is surjective, there exists $\bar{b}$ such that the functional $y \to \nabla_x f_{\bar{b}}(p)^T y$, for $y\in T_{\bm{p}}\manif$, is equal to the functional $y \to (v_1 - u_1)^T y$. This implies that $z = \nabla_x f_{\bar{b}}(p) - v_1 + u_1$ belongs to $(T_p \manif)^\perp$. Furthermore, since
\[
(T_{\bm{p}}\manif)^{\perp} = \Span(\big\{\nabla_x h_1(\bm{p}),\dots,\nabla_x h_m(\bm{p})\big\}),
\]
there exists $w \in \R^m$ such that
\[-z + v_2 -u_2 = 
\sum_{i = 1}^m w_i \nabla_x h_i(\bm{p}).
\]
Hence, by taking the vectors $(\bar{b}, \bm{0}) \in \R^c \times T_{\bm{p}}\manif_J$ we get
\[
\begin{pmatrix}
0 & I_N \\
M(\bm{p}) & B(b,\bm{p})
\end{pmatrix} \begin{pmatrix}
\bar{b} \\
\bm{0}
\end{pmatrix} = 
\begin{pmatrix}
\bm{0} \\
\nabla_x f_{\bar{b}}(\bm{p})
\end{pmatrix}=
\begin{pmatrix}
\bm{0} \\
z+v_1 - u_1
\end{pmatrix}.
\]
Then taking $(u, w, \bm{0}) \in T_{\bm{p}}\manif_J\times \R^{m} \times \R^{\ell}$ gives the expression
\[
\begin{pmatrix}
I_N & 0 & 0\\
H & \nabla_x h_1(\bm{p}) \dots \nabla_x h_m(\bm{p}) & \nabla_x g_{j_1}(\bm{p}) \dots \nabla_x g_{j_{\ell}}(\bm{p})
\end{pmatrix}
\begin{pmatrix}
u \\
w \\
\bm{0}
\end{pmatrix},
\]
which evaluates to
\[
\begin{pmatrix}
u \\
Hu -z + v_2 - u_2
\end{pmatrix} =
\begin{pmatrix}
u \\
-z + u_1 + v_2
\end{pmatrix}
\]
The sum of these two vectors is equal to $(u,v)$, as required.
\end{proof}

For $\hat{J}\subseteq J$, we can see,  from their definitions, that the map $G$ and the submanifold $\psi_{\hat{J}}(\manif_J \times \R^{m} \times \R^{|\hat{J}|})$ capture the defining properties of a critical point, as discussed in Subsection \ref{local_opt_cond_section}. The next result uses this to show the genericity of \eqref{SCC}.

\begin{prop}\label{ssc_prop}
For $J\subseteq [\bar{m}]$, there exists a Lebesgue measure zero set $\mathcal{V}_J\subset\R^c$, such that for every $b \notin \mathcal{V}_J$ every Lagrange multiplier $(\mu_j)_{j \in J}$ of every critical point of $f_b$ on $\manif_J$ is nonzero.
\end{prop}
\begin{proof}
By applying parametric transversality to every subset $\hat{J}$, we get a Lebesgue measure zero set $\mathcal{V}_J$, such that, for $b\in \R^c\setminus\mathcal{V}_J$ the function 
\[
\begin{aligned}
G_b \colon \manif_J &\longrightarrow \manif_J \times \R^{N}\\
\bm{p} &\longmapsto (\bm{p}, \nabla_x f_b(\bm{p}))
\end{aligned}
\]
is transverse to all submanifolds $\psi_{\hat{J}}(\manif_J \times \R^{m} \times \R^{|\hat{J}|})$. Suppose that $\bm{p}$ is a critical point of $f_b$ on $S_J$ and let $\hat{J} \subseteq J$ be the set of indices $j$ such that the Lagrange multiplier $\mu_j$ is nonzero. Then, $\bm{p}$ belongs to $G_b^{-1}(\psi_{\hat{J}}(\manif_J \times \R^{m} \times \R^{|\hat{J}|}))$. Hence, by transversality we have
\[
d_{\bm{p}} G_b\big(T_{\bm{p}} \manif_J\big) + T_{G_b(\bm{p})} \psi_{\hat{J}}(\manif_J \times \R^{m} \times \R^{|\hat{J}|}) = T_{\bm{p}}\manif_J\times \R^{N} \, .
\]
Note that the space on the right hand side has dimension $\dim(\manif) - |J| + N$, while the space on the left hand side has dimension at most $\dim(\manif) +N - 2|J| + |\hat{J}|$. Therefore $\hat{J} = J$.
\end{proof}

\begin{proof}[Proof of Theorem \ref{general_ineqs_thm}]
In order to apply Theorem \ref{KKTfiniteconvergence} to the problem \eqref{POP_general_ineqs}, we need that the local optimality conditions \eqref{LIC}, \eqref{SCC}, and \eqref{SOSC} hold at every global minimizer. By Assumption \ref{LI_assumption}, we have that \eqref{LIC} holds for all global minimizers. By Proposition \ref{SOSC_ineq_prop}, we have that, for $b\notin \mathcal{Z}$, \eqref{SOSC} holds at all global minimizers. Finally, Proposition \ref{ssc_prop} gives a Lebesgue measure zero set $\mathcal{V}_J$ for each $J\subseteq [\bar{m}]$, such that, for $b\notin \mathcal{V}_J$, the local minimizers of $f_b$ on $\manif_J$ satisfy \eqref{SCC}. As we did in the proof of Proposition \ref{SOSC_ineq_prop}, we use the fact that the finite union of Lebesgue measure zero sets is again Lebesgue measure zero, and set $\mathcal{V} = \cup_{J\subseteq [\bar{m}]}\mathcal{V}_J$. Thus, \eqref{SCC} holds for all $b\in \R^c\setminus \mathcal{V}$. Thus, for $b\in \R^c \setminus \big\{\mathcal{Z}\cup\mathcal{V}\big\}$, the assumptions of Theorem \ref{KKTfiniteconvergence} are satisfied, and so the moment-SOS hierarchy applied to \eqref{POP_general_ineqs} has finite convergence. Furthermore, the optimality certificate of flat truncation for the moment hierarchy holds for a sufficiently high order of the hierarchy.
\end{proof}

\begin{proof}[Proof of Corollary \ref{ineqs_cor}]
First note that the space of (homogeneous) polynomials, parameterized by the coefficient space $\R^c$, always has dimension not smaller than the dimension of the manifold $\manif$.

To prove the first claim of the corollary, we fix a point $\bm{p}\in \manif$, and consider the polynomials $r_i(x) = x_i$ for $i$ in $[N]$. Then the vectors $\{\nabla_x r_i(\bm{p})\}_{i=1}^N$ span $\R^N$, which has $T^*_{\bm{p}}\manif$ as a subspace. Hence, the map $\Phi_{\bm{p}}$, as in \eqref{Phi_p_general}, is surjective, and the assumptions of Theorem \ref{general_ineqs_thm} hold.

For the second claim, we again fix a point $\bm{p}\in \manif$ and show the map $\Phi_{\bm{p}}$ is surjective. Since $\bm{p}\neq0$, up to permutation of the coordinates for some $0\leq\ell<N$, we can write $\bm{p}= (0,\dots,0,\bm{p}_{\ell+1} ,\dots, \bm{p}_N)$, such that $\bm{p}_i\neq 0$ for $i>\ell$. Consider the homogeneous polynomials
\[
r_i(x) = \begin{cases}
x_i x_N^{d-1}, \;\;&\text{for} \;\; i\leq \ell\\
x_i^d, &\text{for} \;\; i> \ell.
\end{cases}
\]
The matrix of gradient vectors of the polynomials above, has the following form for $\bm{p}$
\[
\big(\nabla_x r_1(\bm{p})\dots \nabla_x r_N(\bm{p})\big)=
\begin{pmatrix}
\bm{p}_N^{d-1} & &  &  & &\\
 & \bm{p}_N^{d-1} & &  & &\\
&  & \ddots & & &\\
& & & \bm{p}_N^{d-1} & &\\
& & & & d\bm{p}_{\ell+1}^{d-1} &\\
& & & & & \ddots &\\
&  & &  &  & & d\bm{p}_N^{d-1}
\end{pmatrix}.
\]
Hence the vectors $\{\nabla_x r_i(\bm{p})\}_{i=1}^N$ span $\R^N$, and the claim follows as above.
\end{proof}

\section{Further Work}\label{further_work_sec}

We comment briefly on three possible further research directions.

The genericity assumption on the input tensor in Corollary \ref{tensor_finite_converg} may be problematic for some applications. There is an interest in nongeneric families of tensors such as, for example, the matrix multiplication tensor. There, one may use different techniques which do not rely on genericity, but rather use specific features of the given tensor family. The finite convergence result \cite[Theorem 27]{schweighofer_finconverg_2024} may be useful in that setting.

While Theorem \ref{main_result} guarantees convergence at some finite order of the moment-SOS hierarchy, we do not know how high this order may be. As discussed in the introduction, the work \cite{magron_convgrate_2025} proves promising rates of convergence, but does not provide guarantees on the order. There exists extensive work \cite{magron_gradideal_2023,henrion_slowconverg_2025,baldi_effpsatz_2025,baldi_degbnd_2024,nieschw_degbnd_2007} giving upper bounds on the necessary order of the hierarchy. Techniques in those works could be used to give an upper bound on the order of the hierarchy when the feasible set is the product of spheres.

The use of the DPS hierarchy \cite{doherty2004complete} to find the decomposition of a separable state into product states can be formulated as a generalized moment problem (GMP) on the product of complex spheres. As mentioned in Remark \ref{complex_sphere}, our result regarding the product of real spheres can be also applied for problems on the product of complex spheres. This motivates an extension of Theorem \ref{main_result} to instances of the GMP over the product of spheres. 
The proposed method by \cite{li2020separability} relies on the moment approach on the bi-sphere to test the separability of a quantum state in multipartite systems. 
Their approach guarantees a separability certificate at a finite relaxation level, provided the optimal solution meets the flatness condition. By optimizing a generic polynomial, their method encourages this flatness, enabling a separable decomposition and thus a separability certificate.
The work in  \cite{dressler2022separability} builds on this approach by introducing a symmetry argument. This allows them to restrict the bi-sphere to a subset where the bi-sphere points have real and nonnegative first components.
Our preliminary numerical experiments indicate that finite convergence and flat extension seem to occur systematically for random bi-homogeneous quartic objective functions, without needing such restrictions. 

%\bibliographystyle{amsalpha}
%\bibliography{references.bib}
\input{multisphere.bbl}

\end{document}

%% file: multisphere.bbl
\providecommand{\bysame}{\leavevmode\hbox to3em{\hrulefill}\thinspace}
\providecommand{\MR}{\relax\ifhmode\unskip\space\fi MR }
\providecommand{\MRhref}[2]{%
  \href{http://www.ams.org/mathscinet-getitem?mr=#1}{#2}
}
\providecommand{\href}[2]{#2}